\documentclass[reqno]{amsart}

\usepackage{amsmath,amssymb,amsthm,mathrsfs,graphicx}
\usepackage{mathtools}
\usepackage{tikz}
\usetikzlibrary{positioning}
\usepackage{color}
\theoremstyle{plain}
\numberwithin{equation}{section}
\theoremstyle{plain}
\newtheorem{theorem}{Theorem}[section]
\newtheorem{corollary}[theorem]{Corollary}
\newtheorem{lemma}[theorem]{Lemma}
\newtheorem{proposition}[theorem]{Proposition}
\theoremstyle{definition}
\newtheorem{definition}[theorem]{Definition}
\newtheorem{example}[theorem]{Example}
\newtheorem{remark}[theorem]{Remark}

\newtheorem{problem}[theorem]{Problem}

\newcommand{\BC}{{\mathbb C}}\newcommand{\BD}{{\mathbb D}}

\newcommand{\BI}{{\mathbb I}}

\newcommand{\BN}{{\mathbb N}}

\newcommand{\BR}{{\mathbb R}}
\newcommand{\BT}{{\mathbb T}}

\newcommand{\BZ}{{\mathbb Z}}
\newcommand{\cF}{{\mathcal F}}

\newcommand{\cX}{{\mathcal X}}
\newcommand{\cZ}{{\mathcal Z}}
\newcommand{\al}{\alpha}

\newcommand{\de}{\delta}
\newcommand{\vep}{\varepsilon}

\newcommand{\la}{\lambda}

\newcommand{\si}{\sigma}

\newcommand{\Om}{\Omega}
\newcommand{\re}{\operatorname{Re}}

\newcommand{\ov}[1]{{\overline{#1}}}

\newcommand{\inn}[2]{\ensuremath{\langle #1,#2 \rangle}}
\newcommand{\tu}[1]{\textup{#1}}
\newcommand{\wtil}[1]{{\widetilde{#1}}}
\newcommand{\what}[1]{{\widehat{#1}}}

\newcommand{\half}{\frac{1}{2}}
\newcommand{\ands}{\quad\mbox{and}\quad}

\definecolor{purple}{rgb}{.6,0,.7}

\definecolor{green}{rgb}{.0,.6,.0}

\begin{document}

\title{Optimal interpolation in Hardy and Bergman spaces:\ a reproducing kernel Banach space approach}

\author[G.J. Groenewald]{Gilbert J. Groenewald}
\address{G.J. Groenewald, School of Mathematical and Statistical Sciences,
North-West University,
Research Focus: Pure and Applied Analytics,
Private Bag X6001,
Potchefstroom 2520,
South Africa}
\email{Gilbert.Groenewald@nwu.ac.za}

\author[S. ter Horst]{Sanne ter Horst}
\address{S. ter Horst, School of Mathematical and Statistical Sciences,
North-West University,
Research Focus: Pure and Applied Analytics,
Private~Bag X6001,
Potchefstroom 2520,
South Africa
and
DSI-NRF Centre of Excellence in Mathematical and Statistical Sciences (CoE-MaSS),
Johannesburg,
South Africa}
\email{Sanne.TerHorst@nwu.ac.za}

\author[H.J. Woerdeman]{Hugo J. Woerdeman}
\address{H.J. Woerdeman, Department of Mathematics, Drexel University, Philadelphia, PA 19104, USA}
\email{hugo@math.drexel.edu}

\thanks{This work is based on the research supported in part by the National Research Foundation of South Africa (Grant Numbers 118513 and 127364). In addition, HW is partially supported by NSF grant DMS-2000037.}



\begin{abstract}
After a review of the reproducing kernel Banach space framework and semi-inner products, we apply the techniques to the setting of Hardy spaces $H^p$ and Bergman spaces $A^p$, $1<p<\infty$, on the unit ball in $\BC^n$, as well as the Hardy space on the polydisk and half-space. In particular, we show how the framework leads to a procedure to find a minimal norm element $f$ satisfying interpolation conditions $f(z_j)=w_j$, $j=1,\ldots , n$. We also explain the techniques in the setting of $\ell^p$ spaces where the norm is defined via a change of variables and provide numerical examples.
\end{abstract}

\maketitle

\section{Introduction}\label{S:Intro}
The reproducing kernel framework has led to very useful techniques in dealing with Hilbert spaces of functions, especially in interpolation and approximation problems. Based on ideas developed in the early twentieth century (see \cite{aron} for details), the general theory of reproducing kernel Hilbert spaces goes back to the work of Aronszajn \cite{aron0, aron}, followed up by, for instance, \cite{S-NK} and later with a more applied focus in \cite{Aizerman}. A recent comprehensive account can be found in the monograph \cite{PaulsenR}. 

More recently, starting with \cite{Boser}, reproducing kernels have been used extensively in the setting of machine learning; see, e.g., \cite{Fine, Fuku, Gretton, Slavakis, SFL11}. It were actually the applications to machine learning that inspired generalizing the framework to Banach spaces, leading to the notion of reproducing kernel Banach space, as introduced in \cite{ZXZ09}, combined with semi-inner product techniques from \cite{G67}. The advantage of moving out of the Hilbert space setting, is that this allows for a larger variety of norms, enabling one to deal with more intricate problems, while still maintaining a large part of the theory, subject to the Banach space geometry conditions one is willing to accept. One of the seminal results underlying many of the applications is the so-called Representer Theorem, going back to Wahba \cite{W90}, extended to the Hilbert space framework in \cite{SHS01} and to the Banach space setting in \cite{ZXZ09}. It has since been used extensively to address machine learning topics in a Banach space setting, such as Shannon sampling \cite{ZZ11}, multi-task learning problems \cite{ZZ13}, support vector machines \cite{FHY15,XY19} and converse sampling \cite{CM22}, to mention just a few; cf., \cite{LZZ22} for an up to date account and a unified framework incorporating many recent developments. In recent years, much attention has been devoted to extending the framework to non-reflexive Banach spaces, usually built with an $\ell^1$-based construction, cf., \cite{WX,CX21,U21,LWXY23,WXY}, since norms constructed in this way often lead to sparse solutions for the problem under consideration; for an insightful recent review paper and many further references see \cite{X}.
  
A drawback to the reproducing kernel Banach space approach is the additional computational complexity compared to the Hilbert space setting. In particular, when using semi-inner products, the duality operator connecting the semi-inner product on a Banach space $\cX$ and the duality pairing between $\cX$ and its dual space becomes a nonlinear map which in many concrete Banach spaces is not known explicitly, and it is this duality operator that plays an essential role in the solution to optimal interpolation and regularization problems provided by the Representer Theorem. Motivated by this observation, we consider optimal interpolation problems in Hardy $H^p$ spaces and Bergman $A^p$ spaces on various single- and multi-variable domains, and analyze what the reproducing kernel Banach space approach provides in these concrete cases. Some of these problems have been considered as extremal problems in the complex analysis literature since the 1950s, cf., \cite{MR50,RS}, and we draw from techniques developed in these papers to obtain more concrete results and illustrate these with some numerical examples. While the details are worked out for Hardy and Bergman spaces, we hope that our results will also give some insight into how one can approach interpolation problems in other Banach spaces of functions that can be placed in the reproducing kernel Banach space setting; in the last subsection we have already developed the results in the setting of $\ell^p$, where we use a norm depending on a change of basis. In the present paper, we only consider Hardy $H^p$ spaces and Bergman $A^p$ spaces for $1<p<\infty$. Given the recent interests in non-reflexive Banach spaces we plan to return to this topic in a follow-up paper in which the cases $p=1$ and $p=\infty$ are also addressed.

The paper is organized as follows. In Section \ref{S:SIP-RKBS} we review the general theory of semi-inner product reproducing kernel Banach spaces including the first representer theorem which yields optimality conditions for a minimal norm interpolant. In Section \ref{S:SIPsubspace} we consider how the main ingredients of the theory gets translated when restricting to a subspace of a semi-inner product space. In Section \ref{S:Hp interpolation} we provide the details how the Hardy spaces and Bergman spaces appear as semi-inner product reproducing kernel Banach spaces. In Section \ref{S:num} we provide the numerical details how one constructs minimal norm interpolants in $H^p$ on the unit disk  and we also illustrate the results in the setting of $\ell^p$ with a change of basis.

\section{Semi-inner product reproducing kernel Banach spaces: general theory}\label{S:SIP-RKBS}

In this section we give a brief review of some of the theory concerning semi-inner product reproducing kernel Banach spaces. Such spaces play an important role in machine learning theory. The most common examples of reproducing kernel Banach spaces are those generated by a semi-inner product, although there are also examples that do not fall in the semi-inner product framework; we refer to \cite{LZZ22} for a recent overview. We will only consider the semi-inner product framework here, since the special cases we consider later on fit well in this framework, as we  will see in Section \ref{S:Hp interpolation}. Also, we shall only discuss spaces over the complex numbers here.

\subsection{Semi-inner product spaces}\label{SubS:SIP}

The notion of a semi-inner product space goes back to Lumer  \cite{L61}. Most of the material in this subsection originates from the seminal paper of Giles \cite{G67}.

A {\em semi-inner product space} is a vector space $\cX$ (over $\BC$) together with a map $[\cdot ,\cdot ]:\cX\times\cX \to \BC$ such that for all $x,y,z\in\cX$ and $\al\in\BC$:
\begin{itemize}
  \item[(SIP1)] $[x+\al y, z]=[x,z]+\al[y,z]$;
  \item[(SIP2)] $[x,\al y]=\ov{\al}[x,y]$;
  \item[(SIP3)] $[x,x]\geq 0$ and $[x,x]=0$ implies $x=0$;
  \item[(SIP4)] $|[x,y]|^2\leq [x,x] [y,y]$.
\end{itemize}
The difference with an inner product is that a semi-inner product need not be additive in the second component. Indeed, as was shown for instance in \cite[Proposition 6] {ZXZ09},
additivity in the second component, is equivalent to $[x,y]=\overline{[y,x]}$ for all $x,y\in \cX$. Nonetheless, it is easy to see that a semi-inner product also defines a norm $\|\cdot \|$ on $\cX$ in the usual way:
\begin{equation}\label{SIPnorm}
\|x\|=[x,x]^\half,\quad x\in\cX.
\end{equation}
Hence each semi-inner product space is also a normed space. Conversely, by \cite[Theorem 1]{G67}, each normed space $\cX$ admits a semi-inner product $[\cdot,\cdot]$ such that its norm is given by \eqref{SIPnorm}. In that case we say that $[\cdot,\cdot]$ is a semi-inner product on $\cX$ that generates the norm of $\cX$. The existence proof in \cite{G67} relies on the Hahn-Banach theorem, and as a result, in general, there is no construction of the semi-inner product and it need not be unique. We return to this topic in Subsection \ref{SubS:SIPunique}.

Let $(\cX,[\cdot,\cdot])$ be a semi-inner product space with norm $\|\cdot\|$. We then define the dual space $\cX'$ of the normed space $\cX$ in the usual way, and denote the duality pairing between $\cX$ and $\cX'$ by $\inn{\cdot}{\cdot}_{\cX,\cX'}$, i.e.,
\begin{equation}\label{dualpairing}
\inn{x}{\varphi}_{\cX,\cX'}=\varphi(x),\quad x\in\cX,\varphi\in\cX'.
\end{equation}
%When no confusion can arise, we simply write $\inn{\cdot}{\cdot}$ to indicate the duality pairing. 
Using the semi-inner product, each element $x\in\cX$ defines an element $x^{\star_\cX}$ in $\cX'$ via
\begin{equation}\label{Star}
x^{\star_\cX}(y)=[y,x],\quad y\in\cX.
\end{equation}
Note that in case $\cZ$ is a closed complemented subspace of $\cX$, then restricting the semi-inner product of $\cX$ to $\cZ$ provides a semi-inner product on $\cZ$ which generates the norm on $\cZ$ obtained by restricting the norm of $\cX$. However, the dual space of $\cZ$ may not be so straightforward to determine, and as a result the $\star_\cZ$-operation applied to $x\in\cZ$ may be different compared to when $x$ is viewed as an element of $\cX$. The map $x\mapsto x^{\star_\cX}$ connects the semi-inner product and the duality pairing via the following formula:
\begin{equation}\label{SIPvsDP}
[y,x]=\inn{y}{x^{\star_\cX}},\quad x,y\in\cX.
\end{equation}
It is easy to see that $(\la x)^{\star_\cX}=\ov{\la} x^{\star_\cX}$ for any $\la\in\BC$ and $x\in\cX$. However, the $\star_\cX$-operation is additive if and only if the semi-inner product is an inner product. Indeed, since the duality pairing is additive in the second component, it follows directly from \eqref{SIPvsDP} that additivity of the $\star_\cX$-operation corresponds to additivity of the semi-inner product in the second component. To illustrate the above, we consider the semi-inner product of $L^p$; cf., \cite[Section 3]{G67} for more details and proofs.

\begin{example}\label{Ex:Lp}
Let $(\Omega , {\mathcal F},  \mu)$ be a measure space. For $1<p<\infty$, let  $L^p=L^p(\Omega , {\mathcal F},  \mu)$ denote the (equivalence classes of) measurable functions $f$ on $\Omega$  for which the norm 
\[
\|f\|_{L^p}:=\left( \int_\Omega|f|^p \tu{d}\mu \right)^{1/p}
\] 
is finite. The dual space is isometrically isomorphic to $L^q=L^q(\Omega , {\mathcal F},  \mu)$, with $1<q<\infty$ such that $\frac{1}{p}+\frac{1}{q}=1$, with the duality pairing 
\[
\inn{f}{g}= \int_\Omega f g \ \tu{d}\mu,\quad f\in L^p, g\in L^q.
\] 
When $\Omega$ is a topological space and ${\mathcal F}$ is the Borel $\sigma$-algebra we may write $L^p(\Omega ,  \mu)$ instead of $L^p(\Omega , {\mathcal F},  \mu)$. 

As it turns out, by results reviewed in the next subsection, the semi-inner product of $L^p$ is unique, and the $\star_{L^p}$-operation associating the semi-inner product and the duality pairing is given by 
\[
f^{\star_{L^p}}= \frac{1}{\|f\|_{L^p}^{p-2}} \overline{f}|f|^{p-2}.   
\]
Here $\overline{f}(x)=\overline{f(x)}$ and $|f|(x)= |f(x)|$. (Note that in the above formula one needs to interpret $0 |0|^{p-2}$ as $0$ and also understand that ${\bf 0}^{\star_{L^p}}={\bf 0}$.) In fact, the map $f\mapsto f^{\star_{L^p}}$ provides a non-additive (if $p\neq 2$) isometric bijection from $L^p$ to $L^q$, with inverse map the $\star_{L^q}$-operation on $L^q$; it is an interesting exercise to verify directly that $(f^{\star_{L^p}})^{\star_{L^q}}=f$ for each $f\in L^p$. It then follows that the semi-inner product on $L^p$ is given by 
\[
[f,h]=\inn{f}{h^{\star_{L^p}}}=\frac{1}{ \|h\|_{L^p}^{p-2}}\int_\Omega f\overline{h} |h|^{p-2} \tu{d}\mu.
\]  
\end{example}

 We note that the above definition of $\star_{L^p}$  also works for $p=1$. However, in this case the semi-inner product is not unique and we will not pursue the $p=1$ case further in this paper.
%\footnote{HW: In the above examples $p=1$ also works, except for uniqueness of the sip and a lack of an inverse to the star operation. Do we say something about this? By the way, the two examples can in principle be combined but we probably don't want to do this.}

\subsection{Uniqueness of the generating semi-inner product}\label{SubS:SIPunique}

The theory of semi-inner products discussed in the previous section does not require any conditions on the normed space. If one requires more from the semi-inner product, assumptions have to be made about the normed space. In the remainder of this section we shall assume that $\cX$ is a Banach space and discuss how various conditions on the Banach space geometry improve the behaviour of the semi-inner product.

Let $(\cX,\|\cdot\|)$ be a Banach space, with generating semi-inner product $[\cdot,\cdot]$. Define the {\em unit sphere} $S_\cX$ and {\em closed unit ball} $B_\cX^{\rm cl}$ of $\cX$ by 
\[
S_\cX=\{x\in\cX \colon \|x\|=1\}\ands B_\cX^{\rm cl}=\{x\in\cX \colon \|x\|\leq 1\}.
\]  
Now recall that $\cX$ is called {\em smooth} if each $x_0\in S_\cX$ is a {\em point of smoothness} for $B_\cX^{\rm cl}$, that is, if there exists a unique $\varphi\in\cX'$ with $\|\varphi\|=1=\varphi(x_0)$. Moreover, the semi-inner product on $\cX$ is called {\em continuous} if for any $x,y\in S_\cX$:
\[
\re([y,x+ty])\to\re([y,x]) \mbox{ as } \BR\ni t \to 0,
\]
and {\em uniformly continuous} whenever this convergence is uniform in $(x,y)\in S_\cX\times S_\cX$. We then have the following characterizations of uniqueness of the semi-inner product generating the norm of $\cX$; cf., \cite[Theorems 1 and 3]{G67} and their proofs as well as \cite[Theorem 5.4.17]{M98}.

\begin{theorem}\label{T:SIPunique1}
For a Banach space $(\cX,\|\cdot\|)$ and a semi-inner product $[\cdot,\cdot]$ that generates the norm of $\cX$, the following are equivalent:
\begin{itemize}
\item[(1)] $[\cdot,\cdot]$ is the unique semi-inner product that generates the norm of $\cX$;

\item[(2)] $\cX$ is a smooth Banach space;

\item[(3)] $[\cdot,\cdot]$ is a continuous semi-inner product.

\end{itemize}
\end{theorem}

The above theorem does not provide a way to determine the semi-inner product on $\cX$. For this one needs to differentiate the norm of $\cX$. For a given function $h:\cX\to \BC$, the {\em G\^{a}teaux derivative} (or {\em G-derivative} for short) of $h$ at a point $x\in\cX$ in direction $y\in\cX$ is given by
\[
D h(x,y)=\lim_{\BR\ni t\to 0}\frac{h(x+ty)- h(x)}{t}
\]  
in case the limit exist. We say that $h$ is {\em G-differentiable at $x$ in direction $y$} whenever $Dh(x,y)$ exists, $h$ is called {\em G-differentiable at $x$} if it is G-differentiable at $x$ in each direction $y\in\cX$, and $h$ is called {\em G-differentiable} if it is G-differentiable at each $x\in\cX$. Note that the limit in the definition of $Dh(x,y)$ is only over real values of $t$, despite $\cX$ being a complex Banach space. Whenever the limit in the definition of $Dh(x,y)$ is uniform in $(x,y)\in \cX\times \cX$, then we call $h$ {\em uniformly Fr\'{e}chet differentiable}; note that by \cite[Proposition 5.3.4]{AH05}, uniformly Fr\'{e}chet differentiability implies Fr\'{e}chet differentiability.  

Of specific interest in Banach space geometry, are Banach spaces for which the norm is G-differentiable (resp.\ uniformly Fr\'{e}chet differentiable). If that is the case, the Banach space is called {\em G-differentiable} (resp.\ {\em uniformly Fr\'{e}chet differentiable}).

We can now provide another criterion for uniqueness of the semi-inner product, which does provide a formula for the (unique) semi-inner product. 

\begin{theorem}\label{Dnorm}
A Banach space $(\cX,\|\cdot\|)$ has a unique semi-inner product that generates its norm if and only if it is G-differentiable. In that case, the unique semi-inner product on $\cX$ that generates its norm is given by 
\[
[x,y]=\|y\|\left(D\|\cdot\|(y,x)+i D\|\cdot\|(iy,x)\right), \quad x,y\in\cX.
\]
\end{theorem}

The first claim is \cite[Theorem 3]{G67}; the formula for the unique semi-inner product follows from the proof in \cite{G67}. 

In particular, the previous two theorems show that the (unique) semi-inner product of a Banach space is continuous if and only if the norm is G-differentiable. It is also the case that the (unique) semi-inner product of a Banach space is uniformly continuous if and only if the norm is uniformly Fr\'{e}chet differentiable \cite[Theorem 3]{G67}. 

\subsection{Further improvements of the semi-inner product}

In the case of a smooth Banach space $\cX$, or, equivalently, when there exists a continuous semi-inner product generating the norm, variations on orthogonality defined via the norm and semi-inner product also coincide. Recall from \cite{G67} that for $x,y\in\cX$, we say that $x$ is {\em normal} to $y$ and $y$ is {\em transversal} to $x$ if $[y,x]=0$. Note that normality and transversality are not the same, since $[y,x]$ need not be equal to the conjugate of $[x,y]$. Birkhoff \cite{Birkhoff} and James \cite{James} considered a notion of orthogonality in normed spaces defined as:
\[
\mbox{$x$ is {\em orthogonal} to $y$ if } \|x+\la y\|\geq \|x\| \mbox{ for each $\la\in\BC$}.
\]
The following result was proved in \cite[Theorem 2]{G67}.

\begin{theorem}\label{normal}
Let $(\cX,\|\cdot\|)$ be a smooth Banach space. Then for all $x,y\in\cX$,  $x$ is normal to $y$ if and only if $x$ is orthogonal to $y$. 
\end{theorem}

A Banach space $(\cX,\|\cdot\|)$ is called {\em uniformly convex} if for every $\vep>0$ there exists a $\de>0$ such that for all $x,y\in S_\cX$ we have $\|\frac{x+y}{2}\|\leq 1-\de$ whenever $\|x-y\|>\vep$. Also, $\cX$ is called {\em strictly convex} whenever $\|x\|+\|y\|=\|x+y\|$ for $x,y\in\cX\backslash\{0\}$ implies that $y=\la x$ for some $\la>0$. Any uniformly convex Banach space is also strictly convex.

Recall that a Banach space that is either uniformly convex or uniformly Fr\'{e}chet differentiable must be reflexive \cite[Theorem 9.11]{FHHMZ11}. In case the Banach space is uniformly convex as well as smooth, we also have an analogue of the Riesz representation theorem for the semi-inner product \cite[Theorem 6]{G67}. 

\begin{theorem}\label{T:Riesz}
Let $(\cX,\|\cdot\|)$ be a uniformly convex and smooth Banach space. Then for each $\varphi\in\cX'$ there exists a unique $y\in\cX$ such that $\varphi(x)=[x,y]$ for all $x\in\cX$. Moreover, we have $\|\varphi\|_{\cX'}=\|y\|_\cX$.
\end{theorem}

From the uniqueness of the semi-inner product, due to the uniform convexity assumption, it is clear that the vector $y$ must be such that $y^{\star_\cX}=\varphi$. In other words, when the Banach space $\cX$ is uniformly convex and smooth, then  the theorem states that the map $x\mapsto x^{\star_\cX}$ is an isometric bijection mapping $\cX$ onto $\cX'$. 

Finally, in case the smoothness condition in Theorem \ref{T:Riesz} is replaced by the stronger condition that the (unique) semi-inner product is uniformly continuous, then an even more complete duality between $\cX$ and $\cX'$ as semi-inner product spaces is obtained \cite[Theorem 7]{G67}.

\begin{theorem}\label{T:Duality}
Let $(\cX,\|\cdot\|)$ be a uniformly convex Banach space that admits a uniformly continuous semi-inner product. Then $\cX'$ is also uniformly convex with a uniformly continuous semi-inner product which is defined by $[y^{\star_\cX},x^{\star_\cX}]_{\cX'}=[x,y]_{\cX}$.     
\end{theorem}

In terms of the $\star$-operations, under the condition in the theorem, $\cX$ is reflexive and subject to the identification $\cX\simeq \cX''$,  
the theorem implies that the inverse map of $x\mapsto x^{\star_\cX}$ is given by the $\star$-operation on $\cX'$, i.e., $(x^{\star_\cX})^{\star_{\cX'}}=x$ and $(\varphi^{\star_{\cX'}})^{\star_{\cX}}=\varphi
$ for all $x\in\cX$ and $\varphi\in\cX'$. Indeed, for the first identity note that for all $x,y\in\cX$ we have
\begin{align*}
[x,y]_\cX
&=[y^{\star_\cX}, x^{\star_\cX}]_{\cX'} =\inn{y^{\star_\cX}}{(x^{\star_\cX})^{\star_{\cX'}}}_{\cX',\cX''}\\ 
&= \inn{(x^{\star_\cX})^{\star_{\cX'}}}{y^{\star_\cX}}_{\cX,\cX'} =[(x^{\star_\cX})^{\star_{\cX'}},y]_\cX, 
\end{align*}
so that for all $x,y\in\cX$ we have
\[
\inn{x-(x^{\star_\cX})^{\star_{\cX'}}}{y^{\star_\cX}}=[x-(x^{\star_\cX})^{\star_{\cX'}},y]_\cX=0
\]
and we can conclude that $x=(x^{\star_\cX})^{\star_{\cX'}}$. The second identity follows by reflexivity. In conclusion, we have the following result. 

\begin{corollary}
Let $(\cX,\|\cdot\|)$ be a uniformly convex Banach space that admits a uniformly continuous semi-inner product, so that $\cX$ is reflexive. Then the map $x\mapsto x^{\star_\cX}$ is an isometric bijection mapping $\cX$ onto $\cX'$ whose inverse map is given by $\varphi\mapsto \varphi^{\star_{\cX'}}$. 
\end{corollary}

Next we briefly return to the examples of the earlier subsection.

\begin{example}\label{Ex:Lp2}
For $1<p<\infty$, the Banach space $L^p=L^p(\Om,\cF,\mu)$ of Example \ref{Ex:Lp} has a very well behaved Banach space geometry. Indeed, $L^p$ is uniformly convex by \cite[Theorem 9.3]{FHHMZ11} and uniformly Fr\'{e}chet differentiable by \cite[Fact 9.7 and Theorem 9.10]{FHHMZ11} in case $\mu$ is a $\si$-finite measure.     
\end{example}

\subsection{Semi-inner product reproducing kernel Banach spaces} 

With the above preliminaries on semi-inner products out of the way, we turn to reproducing kernel Banach spaces. A Banach space $\cX$ is called a {\em Banach space of functions} on a set $\Om$ if each $f\in\cX$ is a function on $\Om$ and $\|f\|=0$ if and only if $f(z)=0$ for each $z\in\Om$. Hence, point evaluation separates the elements of the Banach space. Following \cite[Definition 1]{ZXZ09}, for a reproducing kernel Banach space one also requires point evaluation to be continuous, and the dual space should be isometrically isomorphic to a Banach space of functions on $\Om$ with the same property. 

\begin{definition}\label{D:RKBS}
A {\em reproducing kernel Banach space} (RKBS for short) is a Banach space of functions $\cX$ on a set $\Om$ such that the dual space $\cX'$ is isometrically isomorphic to a Banach space $\wtil{\cX}$ of functions on $\Om$ and point evaluation is continuous both in $\cX$ and in $\wtil{\cX}$.
\end{definition}

It will be useful to consider an example. 

\begin{example}\label{HpD} Let $1<p<\infty$, $\BD=\{ z \in \BC: |z|<1\}$ and $\BT=\{ z \in \BC: |z|=1\}$. The Hardy space $H^p(\BD)$ is the space of analytic functions $f$ on $\BD$ so that
$$ \| f \|_{H^p} = \sup_{0<r<1} \left( \frac{1}{2\pi} \int_0^{2\pi} | f(re^{it})|^p dt\right)^{\frac1p} < \infty. $$ By taking nontangential boundary limits, one may view $H^p(\BD)$ as a closed subspace of $L^p(\BT)$; for details, see, e.g., \cite{Hoffman}. For $f \in L^p(\BT)$ we let $f(z) =\sum_{j=-\infty}^\infty \hat{f}(j) z^j$ be its Fourier series and $ \hat{f}(j)$, $j\in\BZ$, its Fourier coefficients. Using Fourier coefficients, the Hardy space $H^p(\BD)$ corresponds to the closed subspace of $L^p(\BT)$ consisting of functions $f$ so that $ \hat{f}(j)=0$, $j<0$.  
We will also be using the subspaces
$$ \overline{H^p(\BD)} = \{ f \in L^p(\BT) :  \hat{f}(j)=0, j > 0 \} , zH^p(\BD) = \{ f \in L^p(\BT) :  \hat{f}(j)=0, j \le 0 \}, $$
and view the functions $f \in \overline{H^p(\BD)}$ as co-analytic on the unit disk $\BD$ (i.e., analytic in $\overline{z}$). Thus both $H^p(\BD)$ and $ \overline{H^p(\BD)}$ are Banach spaces of functions on $\BD$.

In order to establish $H^p(\BD) $ as a RKBS we need to view its dual as a Banach space of functions. It is well known (see, e.g., \cite[Chapter 8]{D70}, \cite[Chapter VII]{Koosis}, \cite{Hensgen}) that the dual of $H^p(\BD) $ can be identified as the quotient space $H^p(\BD)'\simeq L^q(\BT)/zH^q(\BD)$, where as usual $\frac1p+\frac1q=1$. Indeed, if we have a linear functional $\varphi \in H^p(\BD)'$, it can be extended (by the Hahn-Banach theorem) to all of $L^p(\BT)$, and therefore it is of the form
$$ \varphi(f) =\varphi_g(f) = \langle f , g \rangle_{L^p(\BT),L^q(\BT)}= \frac{1}{2\pi} \int_0^{2\pi} f(e^{it}) g(e^{it})dt ,\quad f \in H^p(\BD),$$
for some $g \in L^q(\BT)$. The choice of $g$ is not unique as 
\begin{equation}\label{orth} \langle f , h \rangle_{L^p(\BT),L^q(\BT)}=0,\quad f\in H^p(\BD), h \in zH^q(\BD),\end{equation} thus leading to $\varphi_g=\varphi_k$ whenever $k\in g+zH^q(\BD)$. In order to establish $H^p(\BD) $ as a RKBS we will isometrically identify $H^p(\BD)'$ with $\widetilde{H^p(\BD)}$, defined by $ \widetilde{H^p(\BD)}:= \{ g \in L^q(\BT) :  \hat{g}(j)=0, j > 0 \}$ with $\| g \|_{\widetilde{H^p}(\BD)}:= \| \varphi_g \|_{H^p(\BD)'}$. 
In other words, $ \widetilde{H^p(\BD)}$ consists exactly of the functions $g \in \overline{H^q(\BD)}$, but its norm is given by the norm of the linear functional $\varphi_g$ (which is, in general, not equal to the $\overline{H^q(\BD)}$ norm of $g$). We will return to this example (in the multivariable setting) in Subsection \ref{HpB}. 
\hfill $\Box$
\end{example}

Following similar arguments as in the case of a reproducing kernel Hilbert space \cite{aron}, the above definition is strong enough to prove the existence of a reproducing kernel.

\begin{theorem}\label{repkernel} \cite[Theorem 2]{ZXZ09}
Let $\cX$ be a reproducing kernel Banach space on a set $\Om$. Then there exists a unique function $K:\Om \times \Om \to \BC$ such that:
\begin{itemize}
\item[(1)] For every $z\in\Om$, $K(\cdot,z)\in\wtil{\cX}$ and $f(z)=K(\cdot,z)(f)$ for all $f\in \cX$;

\item[(2)] For every $z\in\Om$, $K(z,\cdot)\in\cX$ and $\wtil{f}(z)=f'(K(z,\cdot))$ for any $\wtil{f}\in\wtil{\cX}$ with $f'\in\cX'$ the dual element associated with $\wtil{f}$ via the isometry between $\cX'$ and $\wtil{\cX}$.%\footnote{StH: We should explain better somewhere how this works.}

\item[(3)] The span of $\{K(z,\cdot)\colon z\in\Om\}$ is dense in $\cX$ and the span of $\{K(\cdot,z)\colon z\in\Om\}$ is dense in $\wtil{\cX}$.

\item[(4)] For all $z,w\in\Om$ we have
\[
K(z,w)=K(\cdot,w)(K(z,\cdot))=\inn{K(z,\cdot)}{K(\cdot,w)}_{\cX,\cX'}.
\]
\end{itemize}
\end{theorem}

The function $K$ in the above theorem is called the {\em reproducing kernel} of the RKBS $\cX$. Since the elements of $\cX$ are uniquely determined by their point evaluations, it follows that the reproducing kernel $K$ is uniquely determined. In the case when $\cX=H^p(\BD)$, the reproducing kernel is the Szeg\"o kernel; see Subsection \ref{HpB}.

A more general notion of reproducing kernel Banach spaces is considered in \cite{LZZ22} where the dual space is allowed to be isometrically isomorphic to a Banach space of functions on a different set $\wtil{\Om}$, in which case the reproducing kernel will act on $\Omega \times \wtil{\Omega}$, but we will not pursue that in the present paper.
%\footnote{StH: Depending on how we set things up, we may need two sets. Currently we view $H^{p'}$ as isomorphic to a Banach space of functions on $\BC\backslash \ov{\BD}$, but we could also apply the 'flip' operator to transfer them to functions on $\BD$, I think, though we should be careful that we are still okay with the definition of the pairing, and SIP.}

When $\cX$ is a Banach space of functions on a set $\Om$ with continuous point evaluation, then the challenge is to find another Banach space of functions on $\Om$, with continuous point evaluation, that is isometrically isomorphic to $\cX'$. This is where the semi-inner product plays a role, since it links $\cX'$ to $\cX$ via the $\star_\cX$-operation, provided the Banach space is sufficiently well-behaved.       

\begin{definition}\label{D:SIP-RKBS}
A {\em semi-inner product reproducing kernel Banach space} (SIP-RKBS for short) is a RKBS $\cX$ on a set $\Om$ which is uniformly convex and uniformly Fr\'{e}chet differentiable.
%\footnote{StH: This is the definition in \cite[Page 2753]{ZXZ09}. A question: Can we relax the definition by only assuming $\cX$ is a Banach space of functions on $\cX$ with continuous point evaluation, but not assume this for the space isomorphic to the dual space, as is part of RKBS. In other words, does the existence of a sufficiently nice SIP guarantee that it is a RKBS?}  
\end{definition}

By Theorem \ref{T:Duality}, it follows that the dual of a SIP-RKBS is also a SIP-RKBS. Either by direct proof, as in \cite[Theorem 9]{ZXZ09}, or through the connection between the duality pairing and the semi-inner product, one can prove the following variation on Theorem \ref{repkernel}.

\begin{theorem}\label{T:SIP-RKBS}
Let $\cX$ be a SIP-RKBS on $\Om$. Then there exists a unique function $H:\Om\times \Om\to \BC$ such that $\{H(z,\cdot)\colon z\in\Om \}\subset \cX$ and
\[
f(z)=[f,H(z,\cdot)] \mbox{ for all } f\in\cX,z\in\Om.
\]
Moreover, the reproducing kernel $K$ of the previous section is related through $H$ via $K(\cdot,z)=(H(z,\cdot))^{\star_\cX}$, $z\in\Om$. Furthermore, $\cX'$ (via the identification with $\cX^{\star_\cX}$) is a SIP-RKBS with the unique function $H'$, with above properties of $H$, given by $H'(z,\cdot)=K(z,\cdot)^{\star_\cX}$, $z\in\Om$.
Finally, we have
\[
H(z,w)=[H(z,\cdot),H(w,\cdot)],\quad z,w\in\Omega.
\]
\end{theorem}

Despite creating some possible confusion, we will call the function $H$ in Theorem \ref{T:SIP-RKBS} the {\em SIP-reproducing kernel} of the SIP-RKBS $\cX$.

\subsection{The first representer theorem for SIP-RKBS} 

One of the seminal results on reproducing kernel Banach spaces, which is of great importance in the machine learning literature, is the representer theorem; cf., \cite[Theorems 19 and 23]{ZXZ09}.  In this paper we only consider the first representer theorem (\cite[Theorem 19]{ZXZ09}) and will specify this result to the case of a SIP-RKBS. For the reader's convenience we give a brief sketch of the proof. The setting for the first representer theorem is the following {\em minimal norm interpolation problem}. 

\begin{problem}\label{P:Interpolation} 
Let $\cX$ be a Banach space of functions on a set $\Om$, let 
\[
{\bf z}=\{z_1,\ldots,z_n\}\subset\Om  \ands {\bf s}=\{s_1,\ldots,s_n\}\subset\BC,
\]
with $z_i\neq z_j$ if $i\neq j$. Find an $f\in\cX$ of minimal norm such that $f(z_j)=s_j$ for $j=1,\ldots,n$, if it exists. If so, what is its norm, and when is there a unique solution. 
\end{problem}

In the remainder of this subsection we shall assume $\cX$ to be a SIP-RKBS on a set $\Om$ with SIP-reproducing kernel $H$. There are other variations and more general set-ups; for a recent overview, please see \cite{LZZ22}. 

In order to analyze the above interpolation problem, define the sets
\[
\BI_{\bf z,s}:=\{f\in\cX \colon f(z_j)=s_j,\, j=1,\ldots,n\},\quad
H({\bf z},\cdot)^{\star_\cX}:=\{H(z_j,\cdot)^{\star_\cX} \colon j=1,\ldots,n\}.
\]
Hence, the interpolation problem is to determine whether $\BI_{\bf z,s}$ has a minimal norm element, to find this element and determine whether it is unique. We write $\BI_{\bf z,0}$ in case ${\bf s}=\{0,\ldots,0\}$. Note that for any $x\in \BI_{\bf z,s}$ we have
\begin{equation}\label{Bz0Bzs}
x + \BI_{\bf z,0} = \BI_{\bf z,s}.
\end{equation}
The following lemma provides a criterion on a given set of points ${\bf z}$ under which functions in $\cX$ that satisfy the interpolation conditions exist for each set of values ${\bf s}$.

\begin{lemma}\label{exists} \cite[Lemma 14]{ZXZ09}
Let ${\bf z}=\{z_1,\ldots,z_n\}\subset\Om$. The set $\BI_{\bf z,s}$ is non-empty for each ${\bf s}\in\BC^n$ if and only if the set $H({\bf z},\cdot)^{\star_\cX}$ is linearly independent in $\cX'$.
\end{lemma}

Assuming the condition of the previous lemma, the first representer theorem is the following result; see \cite[Theorem 19]{ZXZ09}. 

\begin{theorem}[First SIP-RKBS Representer Theorem]\label{repr}
Let $\cX$ be a SIP-RKBS on a set $\Omega$ and let $H$ be its SIP-reproducing kernel. Let ${\bf z}$ and ${\bf s}$ as in Problem \ref{P:Interpolation} and assume that the set $H({\bf z},\cdot)^{\star_\cX}$ is linearly independent in $\cX'$. Then there exists a unique $f_\tu{min}\in \BI_{\bf z,s}$ such that 
\[
\|f_\tu{min}\|=\min \{\|f\| \colon f\in \BI_{\bf z,s}\}.
\]
Moreover, for this vector $f_\tu{min}$ we have $f_\tu{min}^{\star_\cX}\in \tu{span} H({\bf z},\cdot)^{\star_\cX}$. 
\end{theorem}

\begin{proof}[Sketch of proof]
From Lemma \ref{exists} we know that $\BI_{\bf z,s}$ is non-empty. The uniform convexity implies that a unique minimal-norm element in $\BI_{\bf z,s}$ exists; see, e.g., \cite[page 53]{istr}. Indicate this vector by $f_\tu{min}$. 

To show that $f_\tu{min}^{\star_\cX}\in \tu{span} H({\bf z},\cdot)^{\star_\cX}$, requires some `orthogonality' arguments.  For subsets $A \subset \cX$ and $B\subset \cX'$, we define
\[
A^\perp = \{ g \in \cX' : \langle a , g \rangle =0 \ \hbox{\rm for all} \ a \in A \}, \ ^\perp B = \{ f \in \cX : \langle f , b \rangle =0 \ \hbox{\rm for all} \ b \in B \}.
\]
Mimicking the Hilbert space proof, since $\cX$ is a SIP-RKBS and hence reflexive, it is easy to see that for any subset $B\subset \cX'$ we have $(^\perp B )^\perp = \overline{\rm span} B$, where $\overline{\rm span}$ indicates the closure of the span of the set it is working on. 

Now observe that by \eqref{Bz0Bzs}, for each $h \in \BI_{\bf z,0}$ and $\lambda\in\BC$ we have $f_\tu{min}+\lambda h\in \BI_{\bf z,s}$ and hence $\| f_\tu{min} \| \le \| f_\tu{min} + \lambda h \|$. In other words, $f_\tu{min}$ is orthogonal to each $h\in \BI_{\bf z,0}$. By Theorem \ref{normal} we have  
\[
0=[h,f_\tu{min}]=\inn{h}{f_\tu{min}^{\star_\cX}} \quad \mbox{for all } h\in \BI_{\bf z,0}. 
\]
Thus $f_\tu{min}^{\star_\cX}\in \BI_{\bf z,0}^\perp$. The reproducing property of $H$ directly yields  
\[
\BI_{\bf z,0}=\prescript{^\perp\!\!}{}{(H({\bf z},\cdot)^{\star_\cX})}.
\]
Hence 
\[
f_\tu{min}^{\star_\cX}\in \BI_{\bf z,0}^\perp= (\prescript{^\perp\!\!}{}{(H({\bf z},\cdot)^{\star_\cX})})^\perp=\overline{\textup{span}}\, H({\bf z},\cdot)^{\star_\cX}=\textup{span}\, H({\bf z},\cdot)^{\star_\cX},
\]
where the second identity follows from the general observation about $(^\perp B )^\perp$ made earlier in the proof. 
\end{proof}

Solving the optimal interpolation problem then works as follows. Note that $f^{\star_\cX}\in \BI_{\bf z,0}^\perp$ does not guarantee that also $f\in \BI_{\bf z,s}$. However, the uniform convexity and linear independence of $H({\bf z},\cdot)^{\star_\cX}$ does guarantee that the intersection of $(\BI_{\bf z,0}^\perp)^{\star_{\cX'}}$ and $\BI_{\bf z,s}$ consists of precisely one element, and this is the unique solution. Hence consider $f \in \cX$ of the form
\[
f^{\star_\cX}=\sum_{j=1}^n c_j H( z_j,\cdot)^{\star_{\cX}}
\]
for $c_1,\ldots,c_n\in\BC$. Then
\[
f=(f^{\star_\cX})^{\star_{\cX'}}=(\sum_{j=1}^n c_j H( z_j,\cdot)^{\star_{\cX}})^{\star_{\cX'}}
\]
depends only on the parameters $c_1,\ldots,c_n\in\BC$, though not in a linear way since the map $f\mapsto f^{\star_{\cX'}}$ is not additive. It then remains to determine the parameters $c_1,\ldots,c_n\in\BC$ from the interpolation conditions $f(z_j)=s_j$, $j=1,\ldots,n$, which will thus typically be nonlinear equations in $c_1,\ldots,c_n$. In case there is only one interpolation condition the formulas simplify. We state this next. 

\begin{corollary}\label{single}
    Let $\cX$ be a SIP-RKBS on a set $\Omega$ and let $H$ be as in Theorem \ref{T:SIP-RKBS}. Let $z_1\in\Om$ and $s_1\in\BC$. Then
    $$ f_\tu{min}(z):= s_1 \frac{H(z_1,z)}{H(z_1,z_1)}$$
    is the unique $f_\tu{min}\in\cX$ such that
\[
f_\tu{min}(z_1)=s_1,\ 
\|f_\tu{min}\|=\min \{\|f\| \colon f\in\cX,\ f(z_1)=s_1\}.
\]

\end{corollary}

\begin{proof}
    By Theorem \ref{repr} 
   there exists a $c_1$ so that $f_\tu{min}^{\star_\cX}=c_1 H(z_1,\cdot)^{\star_\cX}.$
   Applying $\star_{\cX'}$ on both sides we obtain that
$f_\tu{min}=\overline{c_1} H(z_1,\cdot).$ Plugging in the condition that $f_\tu{min}(z_1)=s_1$, gives that $\overline{c_1}= \frac{s_1}{H(z_1,z_1)}$.
\end{proof}

There is also a second representer theorem, which concerns a regularized version of minimizing a loss function among interpolants; details can, for instance, be found in \cite[Section 5.2]{ZXZ09}.

\section{Subspaces of semi-inner product spaces}\label{S:SIPsubspace}

Let $(\cX,\|\cdot\|)$ be a normed space with semi-inner product $[\cdot,\cdot]$ and let $\cZ\subset\cX$ be a closed subspace. Then $\cZ$ becomes a normed space with  semi-inner product, simply by restricting the norm and the semi-inner product. Also, in case $\cX$ is G-differentiable, then so is $\cZ$, so that uniqueness of the semi-inner product carries over from $\cX$ to $\cZ$ (this does not require $\cX$ to be a Banach space; see \cite{G67}). What is less straightforward is what the $\star$-operations on $\cZ$ and $\cZ'$ will be and how they relate to the $\star$-operations on $\cX$ and $\cX'$. In this section, under certain conditions on $\cX$,  we provide formulas for the $\star_\cZ$- and $\star_{\cZ'}$-operations.

Let $y\in\cZ$. Then
\[
y^{\star_\cZ}(x)=[x,y]_\cZ=[x,y]_\cX=y^{\star_\cX}(x),\quad x\in\cZ.
\]
Hence $y^{\star_\cZ}=y^{\star_\cX}|_{\cZ}$. However, there can be many $\varphi\in\cX'$ with $\varphi|_{\cZ}=y^{\star_\cX}|_{\cZ}$ and we require one with the property that $\|y\|_\cZ=\|y^{\star_\cZ}\|_{\cZ'}= \|\varphi\|_{\cX'}$. Such $\varphi\in\cX'$  exists by the Hahn-Banach theorem but it need not be unique. We now recall some results from \cite[Section 4]{RS} that provide conditions under which a unique $\varphi\in\cX'$ with $\varphi|_{\cZ}=y^{\star_\cX}|_{\cZ}$ and $\|y\|_\cZ=\|\varphi\|_{\cX'}$ exists.

Given $\varphi\in\cX'$ we say that $\wtil{\varphi}\in\cX'$ is {\em equivalent to $\varphi$ with respect to $\cZ$} (notation $\wtil{\varphi} \sim_\cZ \varphi$) if
\[
\wtil{\varphi}(z) = \varphi(z) ,\quad z \in \cZ. 
\]
We may just say `equivalent' if the subspace $\cZ$ is clear from the context. It is not hard to see that the set $S_{\varphi, \cZ}:=\{ \wtil{\varphi}  : \wtil{\varphi} \sim_\cZ \varphi \}$ is closed and convex. Clearly, whenever $\what{\varphi} \sim_\cZ \varphi$ we have that
\[
\| \varphi|_{\cZ} \|_{\cZ'} = \| \what{\varphi}|_{\cZ} \|_{\cZ'} \le \inf_{\wtil{\varphi} \in S_{\varphi, \cZ}} \| \wtil{\varphi} \|_{\cX'}.
\]
We say that $ \wtil{\varphi}_* \in \cX'$ is {\em an extremal functional with respect to $\cZ$} associated with $\varphi$ if $\wtil{\varphi}_* \sim_\cZ \varphi$ and
\begin{equation}\label{extremal} 
\| \wtil{\varphi}_*\|_{\cX'} = \inf_{\wtil{\varphi} \in S_{\varphi, \cZ}} \| \wtil{\varphi} \|_{\cX'}. 
\end{equation}
Again, we will leave out ``with respect to $\cZ$'' if no confusion regarding $\cZ$ can occur.

\begin{proposition}\label{P:UniExtremeFunct}
\cite[Theorems II, III and IV]{RS} The extremal functionals associated with $\varphi\in\cX'$ form a nonempty, closed and convex subset of $\cX'$ and $$\inf_{\wtil{\varphi} \in S_{\varphi, \cZ}} \| \wtil{\varphi} \|_{\cX'}=\|\varphi|_{\cZ}\|_{\cZ'}.$$ 
Moreover, if $\cX'$ is strictly convex, then there is exactly one extremal functional associated with $\varphi$.
\end{proposition}

The existence of an extremal functional $\wtil{\varphi}_*$ with $\|\wtil{\varphi}_*\|_{\cX'}=\|\varphi|_{\cZ}\|_{\cZ'}$ is due to the Hahn-Banach theorem. This also means that we may then replace 'inf' by 'min' in the right hand side of \eqref{extremal}.

Let $\varphi\in \cX'$ be nontrivial. We say that $z \in \cZ$ is an {\em extremal point} for $\varphi$ on $\cZ$ if $\| z \|_\cZ = 1$ and $\varphi(z)=\| \varphi|_{\cZ} \|_{\cZ'}$. Recall that the {\em weak topology} on $\cX$ is the weakest topology that makes the linear functionals on $\cX$ continuous (in other words, the smallest topology containing the sets $f^{-1}(U)$, where $f\in \cX'$ and $U\subset \BC$ is open). The Banach space $\cX$ is {\em weakly compact} if the unit ball $B_\cX^{\rm cl}$ in $\cX$ is compact in the weak topology. By Kakutani's Theorem (see \cite[Theorem V.4.2]{C90}) this is equivalent to $\cX$ being reflexive. We have the following.

\begin{proposition}\label{P:UniExtremePoint}
\cite[Parts of Theorems V, VI and VII]{RS} 
Let $\varphi\in \cX'$ be nontrivial and let $\cZ$ be a closed linear subspace of $\cX$. In addition, let $\cX$ be weakly compact. Then the extremal points for $\varphi$ on $\cZ$ form a nonempty, closed and convex subset of $\cZ$. Moreover, if $\cX$ is also strictly convex, then there is exactly one extremal point for $\varphi$ on $\cZ$.
\end{proposition}

Returning to our original problem, given $y\in\cZ$, for $y^{*_{\cZ'}}$ we seek the extremal functional of $y^{*_{\cX'}}$ with respect to $\cZ$. In case $\cX$ is strictly convex, this extremal functional exists and is unique.  

For the remainder of the section we assume that $\cX$ is strictly convex and weakly compact, and that $\cX'$ is also strictly convex. Assume that $\cZ^\perp$ is complemented in $\cX'$, i.e., there exists a bounded projection $P$ of $\cX'$ onto $\cZ^\perp$. Let $\wtil{\cZ}$ be the range of $I-P$, that is, $\wtil{\cZ}$ is the complement of $\cZ^\perp$ in $\cX'$ so that $P$ is the projection on $\cZ^\perp$ along $\wtil{\cZ}$. Note that for the dual space of $\cZ$ we have 
\[
\cZ'\simeq \cX'/\cZ^\perp \simeq \wtil{\cZ}.
\] 
Example \ref{HpD} provides an example of the above equivalences.

We can now describe the $\star$-operations of the semi-inner products on $\cZ$ and $\cZ'$. 

\begin{proposition}\label{P:substar}
Let $(\cX,\|\cdot\|)$ be a uniformly convex Banach space that is weakly compact and has a uniformly continuous semi-inner product $[\cdot,\cdot]$. Let $\cZ$ be a subspace of $\cX$ such that $\cZ^\perp$ is complementable in $\cX'$ with complement $\wtil{\cZ}$ and $P$ the projection onto $\cZ^\perp$ along $\wtil{\cZ}$. Then for $y\in\cZ$ we have
\[
y^{\star_\cZ}=(I-P)[y^{\star_\cX}]\in \wtil{\cZ}.
\]
Moreover, for $\varphi\in \wtil{\cZ}$, 
\[
\varphi^{\star_{\wtil{\cZ}}}=\|\varphi\|_{\wtil{\cZ}} z_*,
\]
where $z_*\in\cZ$ is the unique extremal point of $\varphi$ on $\cZ$.
\end{proposition}

\begin{proof} 
Let $y\in\cZ$. Then $P y^{\star_\cX}\in\cZ^\perp$, so that  
\[
=\inn{z}{P y^{\star_\cX}}=(P y^{\star_\cX})(z)=0,\quad z\in\cZ.
\]
For each $z\in\cZ$ we have
\begin{align*}
[z,y]_{\cZ}
&=[z,y]_{\cX}=\inn{z}{y^{\star_\cX}}_{\cX,\cX'}=\inn{z}{(I-P)y^{\star_\cX}}_{\cX,\cX'} =\inn{z}{(I-P)y^{\star_\cX}}_{\cZ,\wtil{\cZ}}. 
\end{align*}
Since the above identity holds for all $z\in\cZ$, it follows that $y^{\star_\cZ}=(I-P)y^{\star_\cX}$.

Next we turn to the formula for $\varphi^{\star_{\wtil{\cZ}}}$. Note that
\begin{align*}
\|\varphi\|_{\wtil{\cZ}}^2=[\varphi,\varphi]_{\wtil{\cZ}}=\inn{\varphi^{\star_{\wtil{\cZ}}}}{\varphi}_{\cZ,\wtil{\cZ}}
=\varphi(\varphi^{\star_{\wtil{\cZ}}}).
\end{align*}
Thus $\what{z}:=\|\varphi\|_{\wtil{\cZ}}^{-1}\varphi^{\star_{\wtil{\cZ}}}$ is a vector in $\cZ$ such that $\|\what{z}\|=1$, because $\|\varphi^{\star_{\wtil{\cZ}}}\|_\cZ=\|\varphi\|_{\wtil{\cZ}}$, and $\varphi(\what{z})=\|\varphi\|_{\wtil{\cZ}}$. Hence $\what{z}=z_*$ is an extremal point of $\varphi$ on $\cZ$, which is unique according to Proposition \ref{P:UniExtremePoint}, and it follows that $\varphi^{\star_{\wtil{\cZ}}}=\|\varphi\|_{\wtil{\cZ}}\what{z}$ as claimed.
\end{proof}

In case $\cZ'\neq \{0\}$, so that $\cZ^\perp\neq\cX'$, then $\cZ^\perp$ has infinitely many complements in $\cX'$ (assuming $\cZ^\perp$ is complemented in $\cX'$). Hence the projection $P$ is not unique, and so is $y^{*_\cZ}$ in Proposition \ref{P:substar}. However, this nonuniqueness is only caused by the fact that $y^{*_\cZ}$ is in $\wtil{\cZ}$ rather than in $\cZ'$.

\section{Hardy and Bergman spaces as SIP-RKBS}\label{S:Hp interpolation}

In the following subsections, we will show how the theory we have developed so far applies to different Hardy and Bergman space settings.

\subsection{Hardy space on the unit ball in $\BC^n$}\label{HpB}

Let $\BC^n$ be endowed with the usual Euclidean inner product $\langle \cdot , \cdot \rangle_{\rm Eucl}$, and let $B_n$ be the open unit ball in $\BC^n$; thus, $\langle z , w \rangle_{\rm Eucl} = \sum_{i=1}^n z_i \overline{w_i}$ and $B_n=\{ x\in \BC^n : \| x \|_{\rm Eucl} < 1\}$. For $0<r<1$ and a function $f$ with domain $B_n$ we let $f_r$ be the dilated function $f_r(z):=f(rz)$, which is defined on $\frac{1}{r}B_n = \{ x\in \BC^n : \| x \|_{\rm Eucl} < \frac1r \}$. We let $\sigma$ denote the rotation-invariant probability Borel measure on the sphere $S_n=\{ x \in \BC^n : \| x \|_{\rm Eucl} = 1 \}.$ The Hardy space on $B_n$ is the space of holomorphic functions on $B_n$ so that 
$$ \| f \|_{H^p(B_n)} := \sup_{0<r<1} \left( \int_{S_n} |f_r|^p \tu{d} \sigma \right)^\frac1p <\infty.$$ It is well-known (for details, see \cite[Section 5.6]{Rudin}) that one may view $H^p(B_n)$ as a closed subspace of $L^p(S_n,\sigma)$. This requires taking nontangential boundary limits, thus associating with $f\in H^p(B_n)$ a function $f_{\rm boundary}$, which is almost everywhere defined on $S_n$. It will be convenient (justified by the theory) to treat $f$ and $f_{\rm boundary}$ as the same function, as we do for instance in equation \eqref{repH}. 
References on Hardy spaces include \cite{Hoffman, D70, Koosis} for the single variable case and \cite{Rudin} for the several variable case.

The Cauchy kernel for $B_n$ is defined by 
$$ K_{\rm C} (w,z) = \frac{1}{(1-\langle z , w \rangle_{\rm Eucl} )^n} , \ z,w\in B_n,$$
and is the reproducing kernel for $H^p(B_n)$; see \cite[Corollary of Theorem 6.3.1]{Rudin}.
When $n=1$ it is the Szeg\"o kernel $1/(1-z\overline{w})$. 
%\begin{theorem}\label{Rproj}\cite[Corollary of Theorem 6.3.1]{Rudin} Let $1<p<\infty$. For $f\in L^p(S_n,\sigma)$ let
%$$ P_{\rm C}[f](z)= \int_{S_n} K_{\rm C} (z,w) f(w) \tu{d} \sigma(w). $$ Then $f\mapsto P_{\rm C}[f]$ is a bounded linear projection of $L^p(S_n,\sigma)$ onto $H^p(B_n)$.
%\end{theorem}
Thus, when $f\in H^p(B_n)$ we have that
\begin{equation}\label{repH} f(z)= \int_{S_n} K_{\rm C} (w,z) f(w) \tu{d} \sigma(w), z\in B_n.  \end{equation} 

In this context we can state the following new result. We will provide the proof later in this section. 
\begin{theorem}\label{single2} Let $1<p<\infty$, $z_0=(z_0^{(1)}, \ldots , z_0^{(n)} )\in B_n$, and $w_0\in \BC$. The unique minimal norm interpolant $f_{\rm min}$ in $H^p(B_n)$ with the single interpolation condition $f(z_0)=w_0$ is given by
\begin{equation}\label{singleH} f_{\rm min}(z)= w_0 \left( \frac{1-\|z_0\|_{\rm Eucl}^2}{1-\langle z, z_0 \rangle_{\rm Eucl}}  \right)^{\frac{2n}{p}} = w_0 \left( \frac{1-\sum_{j=1}^n|z_0^{(j)}|^2}{ 1-\sum_{j=1}^n z_j \overline{z_0^{(j)}} }  \right)^{\frac{2n}{p}} , \end{equation}
where $z=(z_1,\ldots,z_n) \in B_n$. In addition, $\| f_{\rm min} \|_{H^p(B_n)} = |w_0|(1-\|z_0\|_{\rm Eucl}^2)^\frac{n}{p} $.

\end{theorem}
For $n=1$ the result above was proven before as it follows easily from \cite[Theorem 4]{CMS94}; it can also be found directly in \cite[Theorem 2]{Harmsen}.

%\begin{proof} When $w_0=0$, the results holds trivially (for any $p$), so let us assume that $w_0\neq 0$.
%When $p\in 2\BN$, we have that if $f$ is analytic, then so is $f^\frac{p}{2}$. Also, $\| f \|_{H^p}^p = \| f^\frac{p}{2} \|_{H^2}^2$, so when one is finite, then so is the other. In addition, we have that $\left( \frac{1}{1-\langle z, z_0 \rangle_{\rm Eucl}}  \right)^{\frac{2n}{p}}$ is analytic in $B_n$. The theorem now  follows from Proposition \ref{Bas}.

%When $n=1$ and $1<p<\infty$ the result follows from \cite[Theorem 4]{CMS94} and can also be found directly in \cite[Theorem 2]{Harmsen}.
%\end{proof}

\subsection{Bergman space on the unit ball in $\BC^n$}

Let $\nu$ be the Lebesgue measure on $\BC^n$ normalized so that $\nu(B_n)=1$. The Bergman space $A^p_n$ is defined as $$ A^p_n = L^p(B_n,\nu) \cap \{ f: B_n \to \BC : f \ \hbox{\rm is holomorphic on} \ B_n \}.$$ 
References on the Bergman space include \cite{Axler, DurenBergman, Hedenmalm} for the single variable case and \cite{Rudin} for the several variable case. Extremal functionals in the Bergman setting were considered in \cite{Ferguson2}. The Bergman kernel for $B_n$ is defined by 
$$ K_{\rm B} (w,z) = \frac{1}{(1-\langle z , w \rangle_{\rm Eucl} )^{n+1}} , \ z,w\in B_n, $$ which is the reproducing kernel for $A_n^p$ \cite[Theorem 7.4.1]{Rudin}.
%\begin{theorem}\label{Bproj}\cite[Theorem 7.4.1]{Rudin} Let $1<p<\infty$. For $f\in L^p(B_n,\nu)$ let
%$$ P_{\rm B}[f](z)= \int_{B_n} K_{\rm B} (z,w) f(w) \tu{d} \nu(w). $$ Then $f\mapsto P_{\rm B}[f]$ is a bounded linear projection of $L^p(B_n,\nu)$ onto $A^p_n$.
%\end{theorem}
Thus, when $f\in A^p_n$ we have that
\begin{equation}\label{repB} f(z)= \int_{B_n} K_{\rm B} (w,z) f(w) \tu{d} \nu(w), z\in B_n. \end{equation} 

Similar to the Hardy space on $B_n$, we are able to provide the following solution to the single point minimal norm interpolation problem in the Bergman space.

\begin{theorem}\label{single3} Let $1<p<\infty$, $z_0=(z_0^{(1)}, \ldots , z_0^{(n)} )\in B_n$, and $w_0\in \BC$. The unique minimal norm interpolant $f_{\rm min}$ in $A_n^p$ with the single interpolation condition $f(z_0)=w_0$ is given by
\begin{equation}\label{singleB} f_{\rm min}(z)= w_0 \left( \frac{1-\|z_0\|_{\rm Eucl}^2}{1-\langle z, z_0 \rangle_{\rm Eucl}}  \right)^{\frac{2(n+1)}{p}} = w_0 \left( \frac{1-\sum_{j=1}^n|z_0^{(j)}|^2}{ 1-\sum_{j=1}^n z_j \overline{z_0^{(j)}} }  \right)^{\frac{2(n+1)}{p}} , \end{equation}
where $z=(z_1,\ldots,z_n) \in B_n$. In addition, $\| f_{\rm min} \|_{A^p_n} = |w_0|(1-\|z_0\|_{\rm Eucl}^2)^\frac{n+1}{p} $.
\end{theorem}

%\begin{proof}
%When $p\in 2\BN$, we have that if $f$ is analytic, then so is $f^\frac{p}{2}$. In addition, we have that $\left( \frac{1}{1-\langle z, z_0 \rangle_{\rm Eucl}}  \right)^{\frac{2(n+1)}{p}}$ is analytic in $B_n$. The theorem now  follows from Proposition \ref{Bas}.
%\end{proof}

\subsection{Hardy space on the Polydisk}

Let $\BD = \{ z\in \BC : |z|<1 \}$ be the open unit disk in $\BC$. For $0<r<1$ and a function $f$ with domain $\BD^n$ we let $f_r$ be the dilated function $f_r(z):=f(rz)$, which is defined on $\frac{1}{r}\BD^n$.  The Hardy space on $\BD^n$ is the space of holomorphic functions on $\BD^n$ so that 
$$ \| f \|_{H^p(\BD^n)} := \sup_{0<r<1} \left( \frac{1}{(2\pi)^n}\int_0^{2\pi} \cdots \int_0^{2\pi} |f_r(e^{it_1}, \ldots , e^{it_n})|^p \tu{d} t_1 \cdots \tu{d} t_n \right)^\frac1p <\infty.$$ It is well-known (for details, see \cite{Rudin2}) that one may view $H^p(\BD^n)$ as a closed subspace of $L^p(\BT^n)$.
The Hardy space $H^p(\BD^n)$ is complemented in $L^p(\BT^n)$ (see, e.g., \cite{Ebenstein}), and the reproducing kernel on $H^p(\BD^n)$ is given by
$$ K(w,z)=\prod_{j=1}^n \frac{1}{1-z_j\overline{w_j}} ;$$
see, e.g., \cite[Section 1.2]{Rudin}.

For the single point minimal norm interpolation, we have the following result.

\begin{theorem}\label{single4} Let $1<p<\infty$, $z_0=(z_0^{(1)}, \ldots , z_0^{(n)} )\in \BD^n$, and $w_0\in \BC$. The unique minimal norm interpolant $f_{\rm min}$ in $H^p(\BD^n)$ with the single interpolation condition $f(z_0)=w_0$ is given by
\begin{equation}\label{singleP} f_{\rm min}(z)= w_0 \left( \prod_{j=1}^n \frac{1-|z_0^{(j)}|^2}{1-z_j\overline{z_0^{(j)}}}  \right)^{\frac{2}{p}}  , \end{equation}
where $z=(z_1,\ldots,z_n) \in \BD^n$. In addition, $\| f_{\rm min} \|_{H^p(\BD^n)} = |w_0| \prod_{j=1}^n 
(1-|z_0^{(j)}|^2)^\frac1p $.
\end{theorem}

\subsection{Hardy space on the upper half-plane $\BC_+$}

We recall from \cite[Chapter 11]{D70} the following. Let $\BC_+=\{x+iy : x\in \BR, y >0\}$ be the open upper half-plane of the complex plane. The Hardy space $H^p(\BC_+)$ consists of analytic functions on $\BC_+$ so that $|f(x+iy)|^p$ is integrable for all $y>0$ and 
$$ \| f \|_{H^p(\BC_+)} = \sup_{0<y<\infty} \left( \int_{-\infty}^\infty |f(x+iy)|^p dx \right)^\frac1p <\infty . $$ If $f \in H^p(\BC_+)$ then the boundary function $$f(x) = \lim_{y\to 0+} f(x+iy)$$ exists almost everywhere and belongs to $L^p(\BR, \tu{d}x)$; see \cite[Corollary to Theorem 1.11]{D70}.
By identifying $f \in H^p(\BC_+)$ with its boundary function, we may view $H^p(\BC_+)$ as a subspace of $L^p(\BR, \tu{d}x)$. This is a closed complemented subspace; see \cite[Section 3.6]{Helson} for details.
%The Hilbert of transform is defined by 
%$$ H[f](t) = \frac{1}{\pi} \hbox{\rm p.v.} \int_{-\infty}^\infty \frac{f(\tau)}{t-\tau} \tu{d} \tau , $$
%where p.v. denotes the Cauchy principal value. The Hilbert transform is well-defined on $L^p(\BR, \tu{d} x)$ and defines a bounded linear operator there (Marcel Riesz, 1928). 
The reproducing kernel for $H^p(\BC_+)$ is given by  $$K(w,z)= \frac{1}{2\pi i} \frac{1}{\overline{w}-z}, \ z,w \in \BC_+; $$ see \cite[Theorem 11.8]{D70}.

%\begin{theorem}\cite[Theorem 11.8]{D70} Let $1\le p <\infty$. If $f \in H^p(\BC_+)$ then
%$$ f(z) = \frac{1}{2\pi i} \int_{-\infty}^\infty \frac{f(t)}{t-z} dt, \ {\rm Im} z = y >0,$$
%and the integral vanishes for all $y<0$. Conversely, if $h \in L^p(\BR, \tu{d}x)$ and
%$$\frac{1}{2\pi i} \int_{-\infty}^\infty \frac{h(t)}{t-z} dt \equiv 0, \ {\rm Im} z = y <0,$$ then for $y>0$ this integral represents a function $f \in H^p(\BC_+)$ whose boundary function equals $h$ a.e. on $\BR$.
%\end{theorem}

%For the single point minimal norm interpolation problem, we have the following result.
%
%\begin{theorem}\label{single5} Let $1<p<\infty$, $z_0 \in \BC_+$, and $w_0\in \BC$. The unique minimal norm interpolant $f_{\rm min}$ in $H^p(\BC_+)$ with the single interpolation condition $f(z_0)=w_0$ is given by
%\begin{equation}\label{singleH2} f_{\rm min}(z)= w_0 \left( \frac{2i{\rm Im} z_0}{z-\overline{z_0} }  \right)^{\frac{2}{p}} , \ z \in \BC_+. \end{equation}
%In addition, $\| f_{\rm min} \|_{H^p(\BC_+)} = |w_0|({4\pi{\rm Im} z_0})^\frac{1}{p} $.
%\end{theorem}

%\begin{proof}
%When $p\in 2\BN$, we have that if $f$ is analytic, then so is $f^\frac{p}{2}$. In addition, we have that $\left( \frac{1}{\overline{z_0}-z}  \right)^{\frac{2}{p}}$ is analytic in $\BC_+$. The theorem now  follows from Proposition \ref{Bas}.
%\end{proof}

Analogously, one has the multivariable version on the poly-half-plane $\BC_+^n$. The reproducing kernel for $H^p(\BC_+^n)$ is given by  $$K(w,z)= \frac{1}{(2\pi i)^n}\prod_{j=1}^n  \frac{1}{\overline{w_j}-z_j}, \ z,w \in \BC_+^n. $$ 
See \cite[Section III.5]{SW} for further details. We now obtain 

\begin{theorem}\label{single5n} Let $1<p<\infty$, $z_0 =(z_0^{(1)}, \ldots , z_0^{(n)}) \in \BC_+^n$, and $w_0\in \BC$. The unique minimal norm interpolant $f_{\rm min}$ in $H^p(\BC_+^n)$ with the single interpolation condition $f(z_0)=w_0$ is given by
\begin{equation}\label{singleH2} f_{\rm min}(z)= w_0 \prod_{j=1}^n \left( \frac{2i{\rm Im} z_0^{(j)}}{z_j-\overline{z_0^{(j)}} }  \right)^{\frac{2}{p}} , \ z=(z_1, \ldots, z_n) \in \BC_+^n. \end{equation}
In addition, $\| f_{\rm min} \|_{H^p(\BC_+^n)} = |w_0|\prod_{j=1}^n ({4\pi{\rm Im} z_0^{(j)}})^\frac{1}{p} $.
\end{theorem}

\subsection{Common threads and proofs}

The setting in the above subsections all correspond to 
taking for $\cX$ an appropriate space $L^p(\widehat\Omega, \tu{d} \mu)$, with $1<p<\infty$, and then considering a  subspace of Hardy space type $H^p(\Omega)$ or Bergman space type $A^p(\Omega)$, where we have some Banach space reproducing kernel. In the case of the Hardy space, $\widehat \Omega$ is the distinguished boundary of $\Omega$, while in the Bergman setting we have $\widehat\Omega = \Omega$. We will denote the full space as $L^p$ and the subspace as $\cZ^p$, write $K:\Om \times \Om \to \BC$ for the reproducing kernel of $\cZ^p$ and ${\widetilde{\cZ^p}}$ for the Banach space of functions on $\Om$ isomorphic to $\cZ'$, as in Section \ref{S:SIPsubspace}. Following the terminology of \cite{RS} for $H^p$ on the unit disc ($n=1$), an {\em extremal function} in $\cZ^p$ corresponds to what is called an extremal point in Section \ref{S:SIPsubspace}, while an {\em extremal kernel} in ${\widetilde{\cZ^p}}$ corresponds to an extremal functional in  Section \ref{S:SIPsubspace}. The following applies to all the above settings.

For $g \in {\widetilde{\cZ^p}}$ we define the corresponding linear map $\varphi_g$ on $\cZ^p$ via
$$ \varphi_g(f):= \int_{\widehat\Omega} f(z) g(z) \tu{d} \mu(z) ,\quad  f \in \cZ^p.$$

\begin{proposition}\label{Hpstars} Let $1<p<\infty$ and $\frac1p+\frac1q=1$. Every $\varphi_g, g \in {\widetilde{\cZ^p}}$, has a unique extremal function and a unique extremal kernel. If $f\in \cZ^p$ is the extremal function for $\varphi_g$, then $k: = \| g \|_{{\widetilde{\cZ^p}}} \overline{f} |f|^{p-2}$ is the extremal kernel for $\varphi_g$. In addition, $k^{\star_{L^q}} =  \| g \|_{{\widetilde{\cZ^p}}} f$.
\end{proposition}

\begin{proof} The uniqueness follows from Propositions \ref{P:UniExtremeFunct} and \ref{P:UniExtremePoint}, where we use that $L^p$ and $L^q$ are strictly convex. Next, let $f$ be the extremal function for $\varphi_g$. Then $\| f \|_{\cZ^p} =1$ and $\varphi_g(f)=\| \varphi_g \|_{(\cZ^p)'}$. We now claim that $k\in L^q$ is  the extremal kernel for $\varphi_g$ if and only 
\begin{equation}\label{kernel}
\varphi_k(f)=\| \varphi_g \|_{(\cZ^p)'}= \| k \|_{L^q}.
\end{equation}
Thus
$$ \int_{\widehat\Omega} f(z) k(z) \tu{d}\mu(z)= \int_{\widehat\Omega} |f(z) k(z)| \tu{d}\mu(z)= \|f \|_{L^p} \| k \|_{L^q} = \| k \|_{L^q}. $$
Using the first two equalities, we obtain that
$$ f(x) k (x) \ge 0, |k(x)|^\frac1p = \| \varphi_g \|_{(\cZ^p)'}^\frac1p |f(x)|^\frac1q ,  x \in \widehat\Omega \ \hbox{\rm a.e.} $$
These are exactly the properties that the extremal function and extremal kernel exhibit (see \cite[Theorem 11]{RS} for a particular instance). Note that $k=\| g \|_{{\widetilde{\cZ^p}}} \overline{f} |f|^{p-2}$ has these properties, and therefore it is the unique extremal kernel for $\varphi_g$. Finally, one computes that $k^{\star_{L^q}} =  \| g \|_{{\widetilde{\cZ^p}}} ( \overline{f} |f|^{p-2})^{\star_{L^q}} =  \| g \|_{{\widetilde{\cZ^p}}}f$, where we used that $\| f \|_{\cZ^p}=1$.
\end{proof}

Notice that the above implies that
\begin{equation}\label{fp} fk = \| g \|_{{\widetilde{\cZ^p}}} |f|^p .\end{equation}

\begin{proposition}\label{Kstar} Let $1<p<\infty$. If the reproducing kernel satisfies that $(K( z_1, \cdot))^\frac2p \in \cZ^p$ and $K(z,w)=\overline{K(w,z)}$, then
$$ K(\cdot, z_1)^{\star_{\widetilde{\cZ}^p}}= (K( z_1, \cdot))^\frac2p, $$
or equivalently,
\begin{equation}\label{Kstareq}\left( (K( z_1, \cdot))^\frac2p \right)^{\star_{\cZ^p}}= K(\cdot, z_1). \end{equation}
\end{proposition}

\begin{proof} 
We prove the second equality. Let us start by noting the following:
$$ \| (K( z_1, \cdot))^\frac2p \|_{\cZ^p}^p= \| (K( z_1, \cdot))^\frac2p \|_{L^p}^p = \int_{\widehat\Omega} |K(z_1,w)|^2\tu{d}\mu(w) = $$ $$\langle K(z_1,\cdot ) , K(\cdot , z_1) \rangle = K(z_1,z_1).$$
Next, we perform the following computation:
$$ \left( (K( z_1, \cdot))^\frac2p \right)^{\star_{L^p}}= \frac{1}{\| (K( z_1, \cdot))^\frac2p\|^{p-2}_{L^p}}\overline{(K( z_1, \cdot))^\frac2p} | (K( z_1, \cdot))^\frac2p|^{p-2} =
$$
$$ \frac{1}{ K( z_1, z_1)^{\frac{p-2}{p}}}
\overline{(K( z_1, \cdot))^{\frac2p+\frac{p-2}{p}}} (K( z_1, \cdot))^{\frac{p-2}{p}}= $$
$$ \frac{1}{ K( z_1, z_1)^{1-\frac2p}}
{K(  \cdot, z_1)} (K( z_1, \cdot))^{1-\frac{2}{p}}=: g,$$
where we used that $\overline{K( z_1, \cdot)} = K(\cdot , z_1)$.
To prove \eqref{Kstareq}
we need to show that 
$$ g \sim_{\cZ^p} K(  \cdot , z_1 ). $$
As $\cZ^p = \overline{\rm span} \{K( \zeta , \cdot): \zeta \in \Omega \},  $
it suffices to prove that 
\begin{equation}\label{gK}   
\langle  K( \zeta , \cdot) , g\rangle = \langle  K( \zeta , \cdot) , K( \cdot, z_1) \rangle \end{equation} for all $\zeta \in \Omega$. The right hand side of \eqref{gK} equals $K(\zeta, z_1)$.  
For the left hand side of \eqref{gK} we compute 
$$ \langle  K( \zeta , \cdot), g \rangle = \frac{1}{ K( z_1, z_1 )^{1-\frac2p}}{\int_{\widehat\Omega} {K(w, z_1)} K(\zeta, w)(K( z_1, w))^{1-\frac{2}{p}}}\tu{d}\mu(w)= $$
$$ \frac{1}{ K( z_1, z_1 )^{1-\frac2p}} \langle   K(\zeta, \cdot)(K( z_1, \cdot))^{1-\frac{2}{p}}, K(\cdot, z_1) \rangle = $$
$$ \frac{1}{ K( z_1, z_1 )^{1-\frac2p}} K(\zeta, z_1)(K( z_1, z_1))^{1-\frac{2}{p}} = K(\zeta,z_1),$$
where in the last step we used $K(\zeta, \cdot)(K( z_1, \cdot))^{1-\frac{2}{p}} \in \cZ^p$. Thus we have proven $ g \sim_{\cZ^p} K(  \cdot , z_1 ) $, finishing the proof of \eqref{Kstareq}.
\end{proof}

{\em Proof of Theorems \ref{single2}, \ref{single3}, \ref{single4}, and \ref{single5n}.} Combine Corollary \ref{single} (including the observation $H(z_0,\cdot )^{\star_\cX} = K(\cdot, z_0) $) with Proposition \ref{Kstar}. Note that in all cases $K(\cdot, z_0)$ is bounded away from 0 on $\Omega$, and thus $(K( z_1, \cdot))^\frac2p \in \cZ^p$ holds in all cases.

To compute the norm, we use the reproducing kernel property. Let us illustrate this for Theorem \ref{single3}.
We have
$$ \| f_\tu{min} \|^p_{A_n^p} = |w_0|^p \int_{B_n}
\left| \frac{1-\|z_0\|_{\rm Eucl}^2}{1-\langle z, z_0 \rangle_{\rm Eucl}}  \right|^{2(n+1)}\tu{d}\nu(w) = $$ $$|w_0|^p (1-\|z_0\|_{\rm Eucl}^2)^{2(n+1)} \langle  K_{\rm B} (z_0, \cdot ), K_{\rm B} (\cdot , z_0) \rangle =  $$ $$|w_0|^p (1-\|z_0\|_{\rm Eucl}^2)^{2(n+1)} K_{\rm B} (z_0, z_0 )= |w_0|^p (1-\|z_0\|_{\rm Eucl}^2)^{n+1}. $$ Now take the $p$th root.
\hfill $\Box$

\medskip
Let us also rephrase the first representer theorem (Theorem \ref{repr}) in the current context.

\begin{theorem}\label{repr2}
Consider $\cZ^p$ as before and let $K$ be its reproducing kernel. Let ${\bf z}=\{ z_1, \ldots , z_k \}$, with $z_j\neq z_l$ when $j\neq l$, and ${\bf s}=\{ s_1, \ldots , s_k \}$ be as in Problem \ref{P:Interpolation}. Then there exists a unique $f_\tu{min}\in \cZ^p$ with $f(z_i)=s_i$, $i=1,\ldots , k$, such that 
\[
\|f_\tu{min}\|=\min \{\|f\| \colon f\in \cZ^p\ \hbox{\rm with} \ f(z_i)=s_i, i=1,\ldots , k\}.
\]
Moreover, we have that 
$f_\tu{min}^{\star_{\cZ^p}}\in \tu{span} \{ K(\cdot, z_i), i=1,\ldots , k \}$. 
\end{theorem}

\subsection{Other examples} There are many other settings where the theory developed in this paper can be applied. We provide some here.

\subsubsection{Weighted Bergman space on the disk}
Let $\alpha > -1$. 
By the weighted Bergman space $A^p_\alpha$ on the disk we mean the space of all functions $f$ that are analytic on $\BD$ such that
$$ \|f\|_{A^p_\alpha} := \left( (\alpha+1)\int_\mathbb{D} |f(z)|^p \, (1-|z|^2)^\alpha \tu{d}\nu (z) \right)^{1/p} < \infty. $$
Here the kernel is given by $$K_{{\rm B}, \alpha} (w,z) = \frac{\alpha+1}{(1-\overline{w}z)^{\alpha+2}} \; \; \; \; \; z, w \in \mathbb{D}.$$
The multivariable case is given similarly. For more details, see for instance \cite{DurenBergman}.

\subsubsection{Weighted Bergman space on the right half-plane} Let $\alpha > -1$ and $\BC_{{\rm realpos}}$ be the open right half-plane. The weighted Bergman space $A^p_\alpha (\BC_{{\rm realpos}})$ on the right half-plane is the space of all functions $f$ that are analytic on $\BC_{{\rm realpos}}$ such that
$$ \|f\|_{A^p_\alpha(\BC_{{\rm realpos}})} := \left( \frac{1}{\pi}\int_{\BC_{{\rm realpos}}} |f(x+iy)|^p x^\alpha \, dx \, dy \right)^{1/p}. $$
The reproducing kernel (see \cite{Elliott}) is
$$K(w,z) = \frac{2^\alpha(\alpha+1)}{(\overline{w}+z)^{\alpha+2}} \; \; \; \; \; z,w  \in \mathbb{C}_{\rm realpos}.$$
%\footnote{Let us mention that the weighted Bergman space on the half-plane space is isometrically isomorphic (see \cite{DGM}), via the Laplace transform, to the space $L^2(\mathbb{R}_+, \, d\mu_\alpha)$, where $$ d\mu_\alpha = \frac{\Gamma(\alpha+1)}{2^\alpha t^{\alpha+1}} \, dt, $$ and $\Gamma$ denotes the Gamma function.} 
When $\alpha=0$ we get the traditional Bergman space on the half-plane.  

\subsubsection{Harmonic Bergman functions on half-spaces}

We follow the presentation in \cite{Ramey}. 
The upper half-space $H_n$ is the open subset of $\BR^n $ given by
$$H_n = \{ (x, y) \in \BR^n : x \in \BR^{n-1}, y > 0 \}. $$
Let $p \in (1, \infty ). $ We have let $\tu{d}V$ denote the Lebesgue volume measure on $H_n$. The harmonic Bergman space $b^p(H_n)$ on $H_n$ consists of harmonic functions $u$ on $H_n$ such that
$$ \| u \|_{b^p(H_n)} =\left( \int_{H_n} |u|^p \tu{d}V \right)^\frac1p <\infty .$$ Due to \cite[Proposition 8.3]{ABR} this is a Banach space (of functions) over $\BR$.
Here the reproducing kernel is given by
$$ R(z,w)=\frac{4}{nV(B_{{\rm real},n})} \frac{n(z_n+w_n)^2-|z-\hat{w}|^2}{|z-\hat{w}|^{n+2}},$$
where $\hat{w} = (w_1,\ldots , w_{n-1}, -w_n)$ and $B_{{\rm real},n}$ is the unit ball in $\BR^n$;
see \cite[Theorem 8.22]{ABR}. For $f \in L^p(H_n, \tu{d}V)$ we let
$$ \Pi f (z) = \int_{H_n} f(w) R(z,w)\tu{d} w. $$
We now have that $\Pi : L^p(H_n, \tu{d}V) \to b^p(H_n)$ is a bounded projection \cite[Theorem 3.2]{Ramey}. In addition, due to \cite[Theorem 3.3]{Ramey} we have that the dual space of $b^p(H_n)$ can be identified with $b^q(H_n)$, where as usual $\frac1p+\frac1q = 1$. These results give us the tools to apply the theory of the previous sections (but now over the reals). 

\subsubsection{$\ell^p$, $1<p<\infty$, with a change of basis}\label{lpS}

This example has a different flavor from the previous ones, but still provides an example where the general theory can be applied.

Let $p\in(1,\infty)$ and consider the Banach space
$$ \ell^p=\left\{x=(x_j)_{j=1}^\infty : x_j\in \BC, \| x \|_{\ell^p}:= \bigg( \sum_{j=1}^\infty |x_j|^p \bigg)^\frac1p <\infty \right\}.$$
Let $S: \ell^p \to \ell^p$, $1<p<\infty$, be an invertible map. On $\ell^p$ we define the norm
$$ \| x \|_{\ell^p_S} := \| Sx \|_{\ell^p}. $$
The resulting Banach space we denote by $\ell^p_S$.
The dual space is isometrically isomorphic to $\ell^q_{(S^T)^{-1}}$, with $1<q<\infty$ such that $\frac{1}{p}+\frac{1}{q}=1$, endowed with the 
norm 
$$ \| y \|_{\ell^q_{(S^T)^{-1}}}:= \| (S^T)^{-1} y \|_{\ell^q}.$$
Here $S^T: \ell^q \to \ell^q$ is the transpose of $S$, i.e., $\inn{e_i}{S^T(e_j)}= \inn{Se_i}{e_j}$, where $e_i$ is the sequence with a 1 in spot $i$ and zeroes elsewhere. We will use the shorthand notation $S^{-T} := (S^T)^{-1}$. The 
duality pairing is 
\[
\inn{x}{y}=\sum_{i=1}^\infty x_i y_i ,\quad x\in \ell^p_S, y\in \ell^q_{S^{-T}}.
\] 
The semi-inner product of $\ell^p$ is unique, and the $\star$-operation associating the semi-inner product and the duality pairing is given by 
\[
x^{\star_{\ell^p_S}}= S^T [(Sx)^{\star_{\ell^p}}].
\]
It then follows that the semi-inner product on $\ell^p_S$ is given by 
\[
[x,y]=\inn{x}{y^{\star_{\ell^p_S}}}.
\]  

Let us next illustrate Theorem \ref{repr}. In this case $\Omega=\BN$ and $(H(j,\cdot ))^{\star_\cX}= K(\cdot , j) = e_j$, which is the $j$th standard basis vector. Observe that this yields $$ H(i,\cdot )= S^{-1} [(S^{-T} e_i)^{\star_{\ell^q}}].$$

When we take finite sequences $x=(x_i)_{i=1}^n$ and $S$ an invertible $n\times n$ matrix, we denote the space by $\ell^p_{n,S}$. The above considerations, with obvious minor modifications, all go through for this setting as well.

\section{Optimal interpolation examples and numerical results}\label{S:num}

In this section we illustrate how the theory developed in the previous sections can be used to compute minimal norm interpolants numerically.

\subsection{Optimal interpolation in $H^p(\BD )$, $1<p<\infty$.}In this subsection we show how we can apply the representer theorem for SIP-RKBS to algorithmically solve the optimal interpolation problems in the Hardy space $H^p (\BD )$ on the disk. 

In general, it is hard to determine $ g^{\star_{{\widetilde{H^p(\BD)}}}}$. However, in the setting of the representer theorem with a rational reproducing kernel we may use that $g$ is a rational function. Thus we may apply the results of \cite{MR50}; see also \cite[Section 8.4]{D70} and \cite{RS, K, BK12}. This yields the following procedure.

{\em Procedure to find $ g^{\star_{{\widetilde{H^p(\BD)}}}}$.} We now have the following procedure to find $ g^{\star_{{\widetilde{H^p(\BD)}}}}$ in case $ g \in {\widetilde{H^p(\BD)}}$ has a finite number of poles $\beta_i$, $i=1,\ldots, n$, in ${\mathbb D}$. Let $f= g^{\star_{{\widetilde{H^p(\BD)}}}}$ and $k$ the corresponding extremal kernel for $\varphi_g$, which will have the same poles as $g$. In this case, by \cite[Lemma in Section 8.4]{RS}, we have that
$$ f(e^{it})k(e^{it}) = | \sum_{i=1}^n \frac{c_i}{1-\overline{\beta_i} e^{it}} |^2 \ (= \| g \|_{{\widetilde{H^p(\BD)}}} |f(e^{it})|^p) $$
for some $c_i\in{\mathbb C}$. As an aside, we mention that similar results for the Bergman space have so far been more elusive, as for instance Conjecture 1 in \cite{KS97} shows. Then for $z$ in a neighborhood of ${\mathbb T}$ we have that
$$ f(z) k(z) = \left( \sum_{i=1}^n \frac{c_i}{1-\overline{\beta_i} z} \right)  \left( \sum_{i=1}^n \frac{\overline{c_i}z}{z-{\beta_i} } \right). $$ This yields that 
$$ f(z) = B(z) \left( \frac{\sum_{i=1}^n \frac{c_i}{1-\overline{\beta_i} z}}{B(z)} \right)^\frac2p, $$
where $B$ is the Blaschke product with the same roots as $\gamma(z):= \sum_{i=1}^n \frac{c_i}{1-\overline{\beta_i} z}$. Indeed, the above is derived by observing that the equality  $|\gamma|^2=\| g \|_{{\widetilde{H^p}}} |f|^p$ determines the outer part of $f\in H^p(\BD)$; see also \cite[Formulas (1.3.5-6)]{MR50}.
The unknowns $c_i$ should now be chosen so that $g-f^{\star_{L^p(\BT)}} \in zH^q(\BD)$. 
We thus find that 
$$ g - \frac{1}{\|f \|_{L^p(\BT)}^{p-2} B}  (\overline{\gamma} B )^\frac2p (|\gamma|^\frac2p)^{p-2} = $$
$$ g- \frac{1}{\|f \|_{L^p(\BT)}^{p-2} B} (\overline{\gamma} B )^\frac2p \left( \overline{\gamma} B\right)^{\frac{p-2}{p}} \left( \frac{\gamma}{B} \right)^{\frac{p-2}{p}} 
=  g-\frac{1}{\|f \|^{p-2}_{L^p}}\overline{\gamma}  \left( \frac{\gamma}{B} \right)^{1-\frac2p} \in zH^q(\BD).$$
Using that $\overline{\gamma}(z)= \sum_{i=1}^n \frac{\overline{c_i}z}{z-{\beta_i} }$ and absorbing the factor $\frac{1}{\|f \|_{L^p}}$ in the unknown constants, we need to choose $d_i\in {\mathbb C}$ so that
$$ g(z) - \left( \sum_{i=1}^n \frac{\overline{d_i}z}{z-{\beta_i} } \right) \left( \frac{1}{B(z)} \sum_{i=1}^n \frac{d_i}{1-\overline{\beta_i} z} \right)^{1-\frac2p} \in zH^q(\BD) , $$
where $B$ is the Blaschke product with the same roots as $\sum_{i=1}^n \frac{d_i}{1-\overline{\beta_i} z}$.
In other words, the residues of $\frac{g(z)}{z}$ and 
$$\frac{k(z)}{z}:= \left( \sum_{i=1}^n \frac{\overline{d_i}}{z-{\beta_i} } \right) \left( \frac{1}{B(z)} \sum_{i=1}^n \frac{d_i}{1-\overline{\beta_i} z} \right)^{1-\frac2p}$$ at the poles $\beta_i$ should coincide.
We then have that $g^{\star_{{\widetilde{H^p(\BD)}}}} = f=k^{\star_{L^q(\BT)}}$ and 
$$\| g \|_{\widetilde{H^p(\BD)}} =\| \varphi_g\|_{H^p(\BD)'} = \| \varphi_k \|_{H^p(\BD)'} = \| k \|_{L^q}= \left( \frac{1}{2\pi} \int_0^{2\pi} |\sum_{i=1}^n \frac{d_i}{1-\overline{\beta_i} e^{it}}|^2 dt \right)^{\frac1q} . $$

For the case when there are several interpolation conditions we can recover the procedure from \cite{CMS94} by applying the above results as follows.

We seek $f \in {H^p(\BD)}$ of minimal norm so that $f(z_i)=w_i$, $i=1,\ldots , n$. Call the optimal function $f_{\rm min}$.
By Theorem \ref{repr} we know that $f_{\rm min}^{\star_{H^p(\BD)}} \in {\rm span} \{  \frac{1}{1-\overline{z_i}z} : i=1,\ldots, n \}$. Using now the procedure from the previous section, we arrive at the following.

{\em Procedure for finding the minimal norm interpolant in $H^p(\BD)$:}

\noindent For $d_i\in\BC, i=1, \ldots ,n$, consider $$f(z)=B(z) \left( \frac{\sum_{i=1}^n \frac{d_i}{1-\overline{z_i}z}}{B(z)} \right)^\frac2p,$$ where $B$ is the Blaschke product with the same roots as $\sum_{i=1}^n \frac{d_i}{1-\overline{z_i}z}.$ We now need to determine the values of $d_i \in \BC$, $i=1,\ldots, n$, so that  $f(z_i)=w_i$, $i=1,\ldots , n$. The resulting function $f$ equals $f_{\rm min}=f_{\rm min}^{[p]}$, where we added the superscript $[p]$ to emphasize that $f_{\rm min}$ depends on $p$.

When we apply the above procedure with $z_1=\frac12, z_2=-\frac13, z_3=\frac{i}{4}$, $w_1=1,w_2=0.9, w_3=0.8$ and letting $p\in[1.7,2.6]$ we obtain the graph in Figure \ref{fig:enter-label0} for $\| f_{\rm min}^{[p]} \|_{H^p(\BD)}$ as a function of $p$: 
\begin{figure}[h]
    \centering
   \includegraphics[width=5in, height=3in]{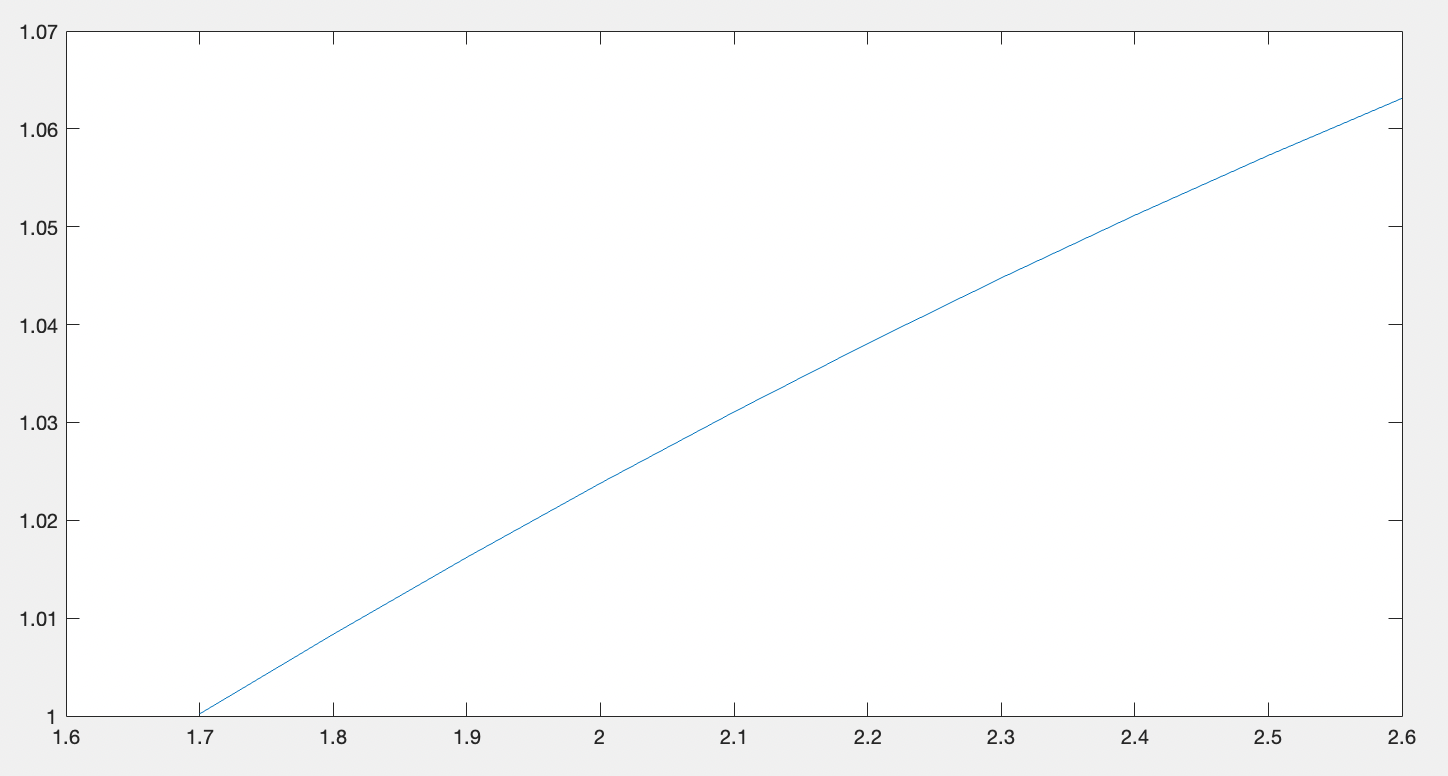}
    \caption{Value of $\| f_{\rm min}^{[p]} \|_{H^p(\BD)}$ as a function of $p$.}
    \label{fig:enter-label0}
\end{figure}

\begin{remark}
If $\mu$ is a finite measure, $1<p\le r<\infty$ and $f\in L^p(\BT)=L^p(\BT) \cap L^r(\BT) $,
we have that 
\begin{equation}\label{prineq} \| f \|_{L^p(\BT) } \le \mu(\Omega)^{\frac1p -\frac1r} \| f \|_{L^r(\BT) }. \end{equation} In the above we have that $\mu$ is a probability measure. 
Now, as both  $f_\tu{min}^{[p]}$ and  $f_\tu{min}^{[r]}$ are interpolants, we obtain for $p\le r$ that
$$ \| f_\tu{min}^{[p]} \|_{H^p(\BD)} \le \| f_\tu{min}^{[r]} \|_{H^p(\BD)} \le  \| f_\tu{min}^{[r]} \|_{H^r(\BD)},  $$ where in the first inequality we use that $f_\tu{min}^{[p]}$ is the optimal solution in the $H^p(\BD)$ norm and in the second inequality we use \eqref{prineq}.  This explains why the graph in Figure \ref{fig:enter-label0} is increasing.
\end{remark}

Finally, let us mention that we only let $p$ vary in the interval $[1.7,2.6]$ as the numerical results were less reliable beyond this region. We suspect that this is due to some of the roots (needed for $B(z)$) approaching the boundary of the disk when $p$ moves outside this region. To analyze this in more depth will be a separate project.

\subsection{Optimal interpolation in $\ell^p_{n,S}$.}

We use the finite dimensional version of Subsection \ref{lpS}.

We have that $\Omega = \{ 1, \ldots , n \}$ and we choose $z_1, \ldots, z_k \in \Omega$. For notational convenience, let us assume that $z_1=1, \ldots , z_k=k$ are the interpolation points. Thus we are looking for a minimal norm solution to $x\in \ell^p_{n,S}$ with $x(j)=s_j$, $j=1, \ldots , k$. Theorem \ref{repr} now tells us that the optimal solution satisfies
$$x_\tu{min}^{\star_\cX}=c_1 e_1+ \cdots + c_k e_k,$$
for some $c_1,\ldots , c_k \in {\mathbb C}$. In this case, we have
\begin{equation}\label{xmin} x_\tu{min} = S^{-1}[(S^{-T}(c_1 e_1+ \cdots + c_k e_k))^{\star_q}].\end{equation} Note that when $q$ is an even integer, the equations the interpolation conditions give are polynomial of total degree $q-1$ in $c_j$ and $\overline{c_j}$, $j=1,\ldots, n$.
%If we take for instance $p=\frac43$, then $q=4$, which then gives us that
%$$x_\tu{min} = S^{-1}[(S^{-T}(d_1 e_1+ \cdots + d_k e_k))\circ (S^{-T}(d_1 e_1+ \cdots + d_k e_k)) \circ \overline{(S^{-T}(d_1 e_1+ \cdots + d_k e_k))}],$$ where conveniently we have absorbed the division by the norm in the unknown coefficients $c_j$ and renamed them $d_j$, $j=1,\ldots k.$ Insisting now that $x_\tu{min}(j)=s_j$, $j=1,\ldots , k,$ yields $k$ with $k$ unknowns $d_j$, $j=1,\ldots ,k$. The unknowns appear in a cubic way.

To illustrate the above numerically, we implemented the above procedure in MATLAB as follows. This involved writing a function that takes as input  $c_1,\ldots ,c_k$, and then it produces $ (x_\tu{min})_1-s_1, \ldots , (x_\tu{min})_k-s_k$, where $x_\tu{min}$ is as in \eqref{xmin}. Subsequently, we use the MATLAB command 'fsolve' to find the $c_1,\ldots ,c_k$ that produces all zeros as the outcome of the aforementioned function.
For the finite dimensional case with $n=4$,
$$ S = \frac14 \begin{bmatrix} 1 & 1 & 1 & 1 \cr 1 & -1 & 1 & -1 \cr 1 & 1 & -1 & -1 \cr 1 & -1 & -1 & 1 \end{bmatrix} $$ and interpolation conditions given by $k=3$ and $s_1=1, s_2=2, s_3=3$, we find the results below. The first graph (in Figure \ref{fig:enter-label}) has on the $x$-axis the value of $p$ and on the $y$-axis the value of $(x_{\rm min}^{[p]})_4$. The lowest value of $p$ is $1.02$.
\begin{figure}[h]
    \centering
   \includegraphics[width=5in, height=3in]{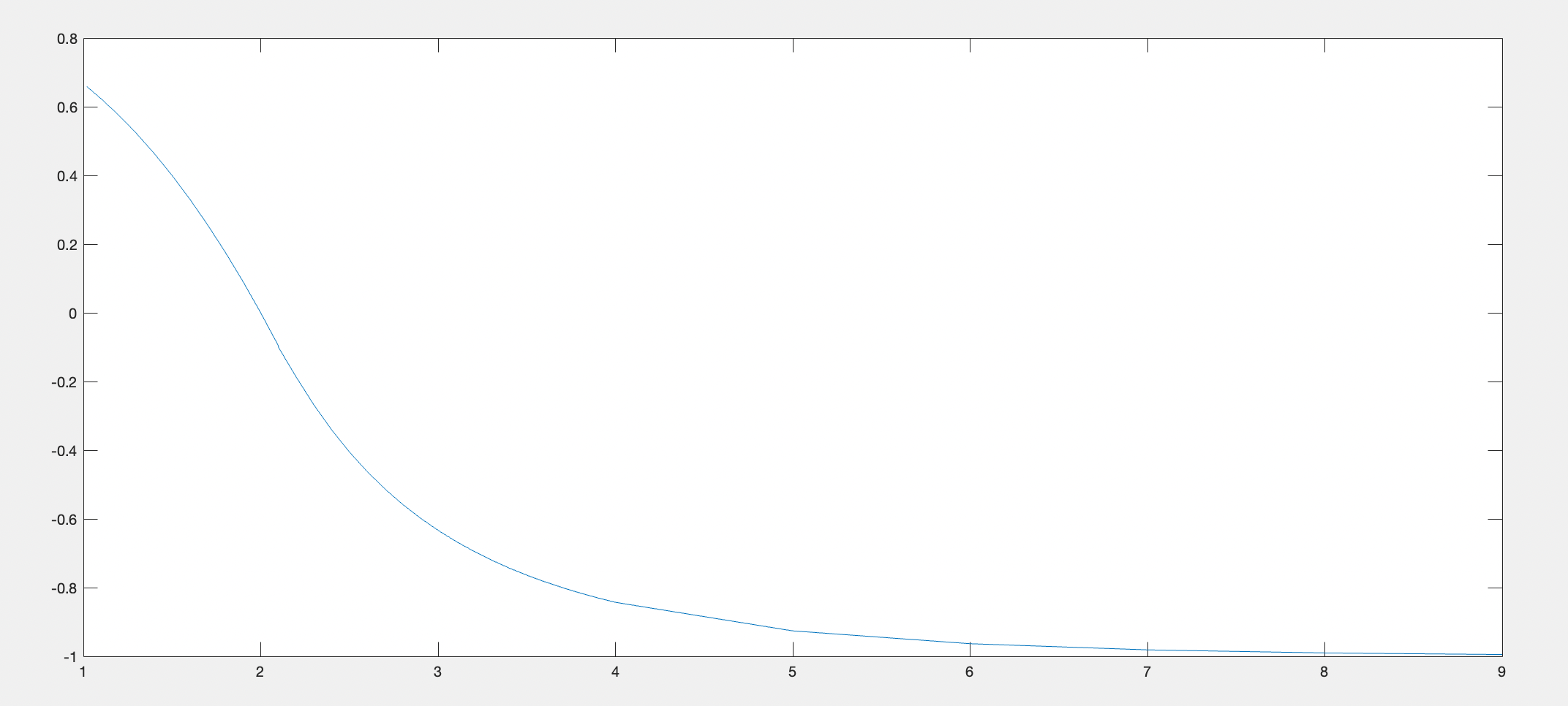}
    \caption{Value of $(x_{\rm min}^{[p]})_4$ as a function of $p$.}
    \label{fig:enter-label}
\end{figure}
The second graph (in Figure \ref{fig:enter-label2}) shows the minimal norm of the minimal interpolant as a function of $p$.
\begin{figure}[h]
    \centering
   \includegraphics[width=5in, height=3in]{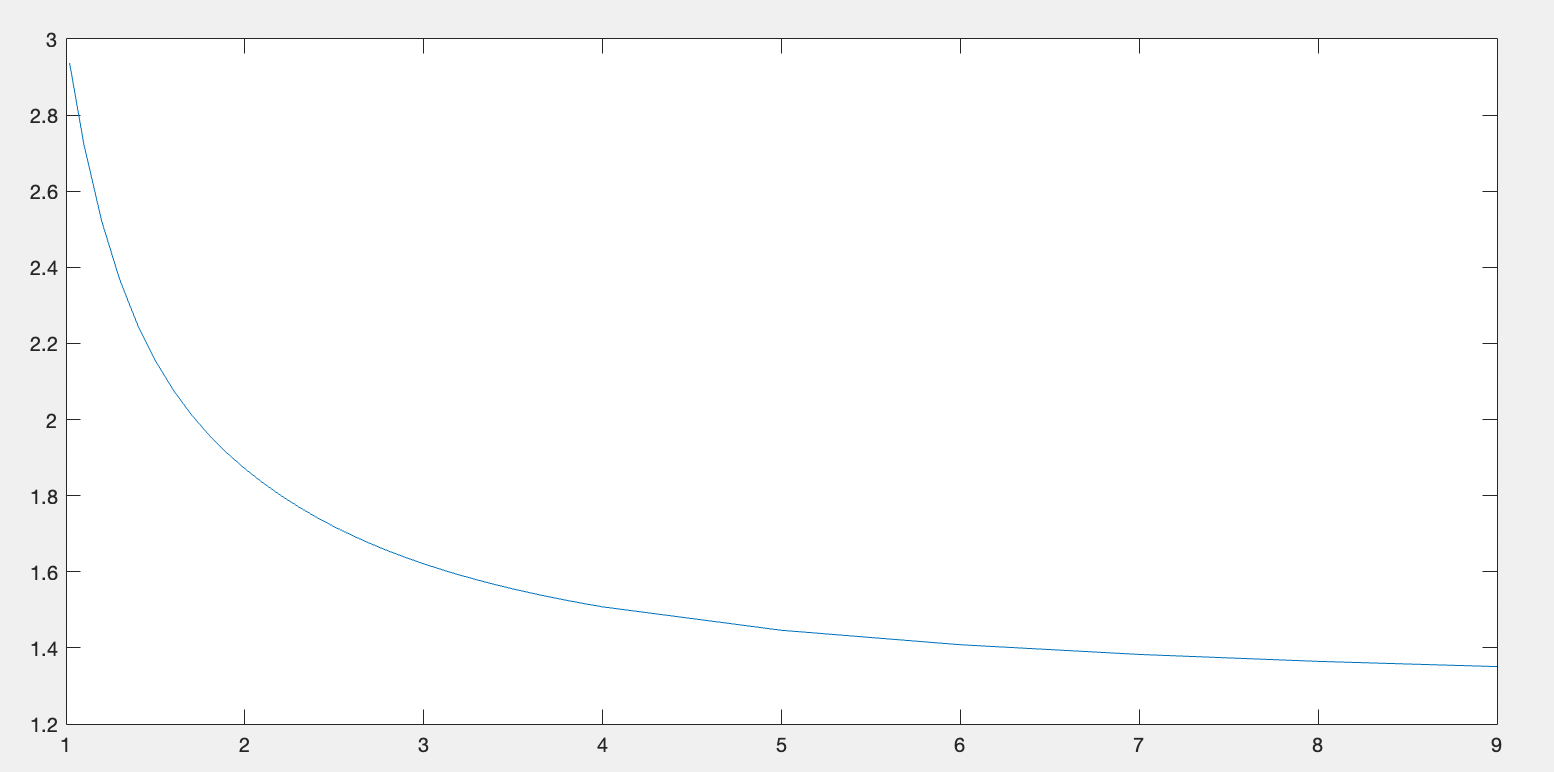}
    \caption{Value of $\| x_{\rm min}^{[p]} \|_{\ell^p_{4,S}}$ as a function of $p$; here $S$ is $4 \times 4$.}
    \label{fig:enter-label2}
\end{figure}

We note that for $p=1$, any value in $[0,4]$ is a possible $(x_{\rm min}^{[p]})_4$ (not unique!), and the optimal norm is 3.5. When $p=\infty$, we have 
$(x_{\rm min}^{[p]})_4=-1$ and the optimal norm is 1.25.

Next we took a $16\times 16$ matrix $S$ so that $S^{-1}$ is the tridiagonal Toeplitz with main diagonal entries equal to 2, and the sub/super-diagonal entries equal to 1. The interpolation conditions are $x(1)=1, x(6)=2, x(11)=3$. We let $p$ range from 1.05 to 10. The results are in Figure \ref{fig:enter-label3}
\begin{figure}[h]
    \centering
   \includegraphics[width=5in, height=3in]{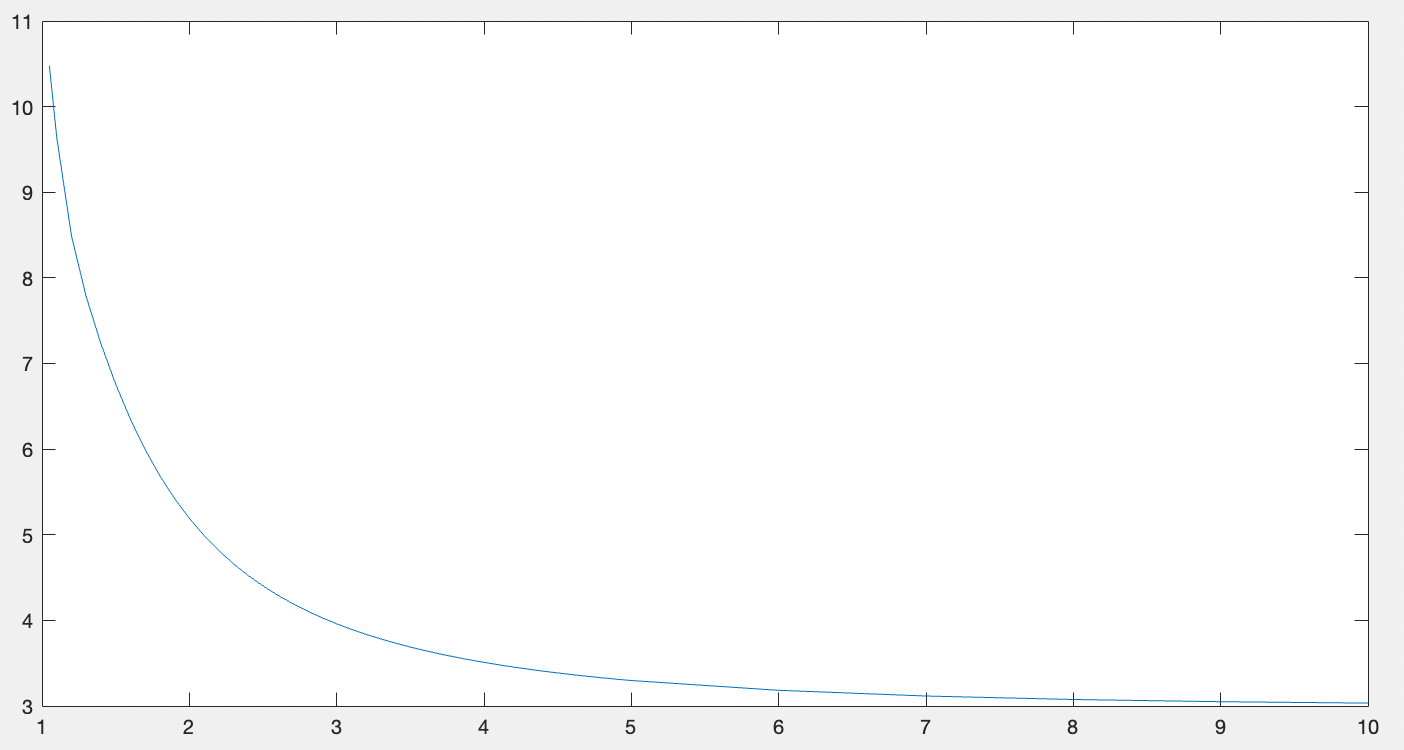}
    \caption{Value of $\| x_{\rm min}^{[p]} \|_{\ell^p_{16,S}}$ as a function of $p$; here $S$ is $16 \times 16$.}
    \label{fig:enter-label3}
\end{figure}

\begin{remark}
In general we have that 
 $1<p\le r<\infty$ implies $\| x\|_{\ell^p} \ge \| x \|_{\ell^r}. $
Thus, as both  $x_\tu{min}^{[p]}$ and  $x_\tu{min}^{[r]}$ are interpolants, we obtain for $p\le r$ that
$$ \| S x_\tu{min}^{[p]} \|_{\ell^p} \ge \| S x_\tu{min}^{[p]} \|_{\ell^r} \ge  \| S x_\tu{min}^{[r]} \|_{\ell^r}. $$ This explains why the graphs in Figures \ref{fig:enter-label2} and \ref{fig:enter-label3} are decreasing.
\end{remark}

The graphs in Figures \ref{fig:enter-label2} and \ref{fig:enter-label3} suggest that the optimal value is a continuous function of $p$. In fact, $x_\tu{min}^{[p]}$ depends continuously on $p$ (as Figure \ref{fig:enter-label} suggests) by the Maximum Theorem, as we show next.

\begin{lemma}
$x_\tu{min}^{[p]}$ and $\| x_\tu{min}^{[p]} \|_{\ell^p_{n,S}}$ depend continuously on $p$.    
\end{lemma}

\begin{proof} 
Let $M>0$ be so that 
$$ \sup_{1<p<\infty} \| S x_\tu{min}^{[p]} \|_{\ell^1_n} <M.$$
For instance, one may choose $M=\| S (\sum_{j=1}^k s_je_j )\|_{\ell^1_n} + 1$. Let $$C=\{ x=(x_r)_r : x_j=s_j, j=1\dots, k, \| x \|_{\ell^1_n} \le M \}.$$ Next, let
$$ f(x,p) = \| x \|_{\ell^p_{n,S}} , f^*(p)= \min_{x\in C} f(x,p) . $$
Applying now the Maximum Theorem (see, e.g., \cite[Theorem 2, page 116]{Berge}) we obtain that 
$$ C^*(p)=\{ x \in C : f(x,p) = f^*(p) \} $$
is continuous in $p$. In our case, we have $C^*(p) = \{ x_\tu{min}^{[p]} \}$, and thus the lemma follows. As a composition of continuous functions is continuous, the continuity of $\| x_\tu{min}^{[p]} \|_{\ell^p_{n,S}}$
also follows.
\end{proof}

\subsection{Using $\ell^p_{N,S}$ norm minimization in time delay estimation}

This subsection is based on the papers \cite{MN,ZSZ}. We start with a description of the problem as outlined in \cite{ZSZ}. This concerns observed discrete-time signals at two sensors
$$ \left\{ \begin{matrix} x_1[n] = s[n] + v_1[n] ,  \ \ \ \ \ \ \ \  & n=1,2\ldots , N, \cr
x_2[n] = \beta s[n-D] + v_2[n] , &n=1,2\ldots, N.\end{matrix} \right.$$
where $s[n]$ is an unknown source signal, $\beta$ and $D$ are the attenuation factor and time delay between the sensors, respectively, and $v_1[n]$ and $v_2[n]$ are uncorrelated additive noises independent of the source signal. In order to estimate the time delay $D$, the following minimization is used:
\begin{equation}\label{hopt} h_{\rm opt} = {\rm arg \ min}_h \sum_{n=M+1}^{N-M} \left| x_2[n] - \beta \sum_{n=-M}^M h_i x_1[n-i]\right|^p,\end{equation}
where the approximation 
$$ s[n-D]=\sum_{i=-\infty}^\infty s[n-i]{\rm sinc}(i-D)\approx \sum_{i=-M}^M s[n-i]{\rm sinc}(i-D)$$ is used. The delay $D$ is next approximated by computing
$$D_{\rm opt} ={\rm arg \ max}_D \sum_{i=-M}^M h_{{\rm opt},i}{\rm sinc}(D-i).$$ Depending on the nature of the noise some values of the parameter $p$ perform better than others. When the noise is Gaussian, $p=2$ is a good choice, while in non-Gaussian instances a value $1<p<2$ may perform better. 

We can recast the problem in the setting of this paper as follows. Let $T$ be the (tall) Toeplitz matrix
$$ T= (x_1[i-j])_{i=2M+1, j=0}^{N \ \ \ \ \ \ \ \ 2M}$$
and $y=(x_2[i])_{i=M+1}^{N-M}$, and perform Gaussian elimination to write $[T\ y]$ as
$$[T\ y]=S[E \ \hat{y}],$$ with $S$ invertible and $E$ the row reduced echelon form of $T$. Let $r$ be the rank of $T$.  We now have that \eqref{hopt} is equivalent to 
\begin{equation}\label{hopt2} x_{\rm min} = {\rm arg \ min}_x  \| Sx\|_p \ {\rm subject \ to} \ x_{j}=\hat{y}_j, j=r+1,\ldots , N  ,\end{equation}
and $Eh_{\rm opt}=\hat{y}-x_{\rm min}$. Notice that the last equation is easy to solve, as $E$ is in row reduced echelon form. 

As an illustration, we took a source signal $s[n]$, $n=1, \ldots, 2001$, with each entry a random number in the interval $[0,1]$ (using 'rand' in Matlab). In addition, we took $\beta=1$, $D=5$ and non-Gaussian noise with values $v_2[n] \in \{ -0.4,0,0.4 \}$ and  $v_2[n] \in \{ -10,0,10 \}$. We put $M=10$. Letting $p=1.01$, we find the solutions 
$$ h_{{\rm opt},p=1.01}^{0.4} = 
\begin{bmatrix}
  -0.000232823052903 \cr 
   0.000040776761574 \cr 
  -0.000049481864553\cr 
   0.000133215025281\cr 
  -0.000183737404681\cr 
   1.000000881642775\cr 
  -0.000044649227715\cr 
  -0.000196161795022\cr 
  -0.000095126313718\cr 
  -0.000097099554212\cr 
   0.000033868802636\cr 
  -0.000146470233426\cr 
  -0.000200182922050\cr 
  -0.000079300151792\cr 
   0.000079332208475\cr 
   0.000120242785738\cr 
   0.000065015502363\cr 
   0.000281342918683\cr 
  -0.000004940174283\cr 
   0.000070394035326\cr 
   0.000007260964693 \end{bmatrix},
  h_{{\rm opt},p=1.01}^{10}=\begin{bmatrix}    -0.005820061605867 \cr
   0.001019548250559 \cr
  -0.001236528515010 \cr
   0.003331180840021 \cr
  -0.004592569849389 \cr
   1.000022762279084 \cr
  -0.001116623684563 \cr
  -0.004903685210011 \cr
  -0.002378077978435 \cr
  -0.002427966380243 \cr
   0.000845998083710 \cr
  -0.003661436517994 \cr
  -0.005004189106320 \cr
  -0.001983061497221 \cr
   0.001983807240446 \cr
   0.003005241842425 \cr
   0.001625326145315 \cr
   0.007033634145231 \cr
  -0.000123313898717 \cr
   0.001759697717274 \cr
   0.000181776684826 \end{bmatrix}, $$
   where the superscript in $h_{\rm opt}$ indicates the level of the noise.
   From this it is clear that we retrieve $D=5$ for both noise levels.
   When we use $p=2$ we find 
$$ h_{{\rm opt},p=2}^{0.4} = 
\begin{bmatrix}
 0.018009040220347 \cr
   0.002332350834668 \cr
   0.001890133545715 \cr
  -0.003072754337702 \cr
   0.030593903501315 \cr
   0.988017340867130 \cr
  -0.021206323141440 \cr
   0.017823418513643 \cr
   0.015038119392389 \cr
   0.002017146896049 \cr
  -0.001601106447733 \cr
   0.012924523085732 \cr
  -0.000953916946889 \cr
  -0.008799966349236 \cr
   0.001692369555610 \cr
  -0.016685275281673 \cr
  -0.012582988810589 \cr
  -0.023738009359622 \cr
  -0.010928158128446 \cr
  -0.011105474647007 \cr
   0.009191378729438 
\end{bmatrix}, 
h_{{\rm opt},p=2}^{10} = 
\begin{bmatrix}  
   0.450226005508666 \cr
   0.058308770866700 \cr
   0.047253338642885 \cr
  -0.076818858442547 \cr
   0.764847587532876 \cr
   0.700433521678247 \cr
  -0.530158078536007 \cr
   0.445585462841087 \cr
   0.375952984809734 \cr
   0.050428672401218 \cr
  -0.040027661193312 \cr
   0.323113077143296 \cr
  -0.023847923672227 \cr
  -0.219999158730899 \cr
   0.042309238890246 \cr
  -0.417131882041814 \cr
  -0.314574720264717 \cr
  -0.593450233990547 \cr
  -0.273203953211154 \cr
  -0.277636866175162 \cr
   0.229784468235959 \cr
   \end{bmatrix} .$$
   At the noise level 0.4, we still conclude $D=5$ when we use $p=2$, but for the noise level 10 the optimization with $p=2$ no longer gives a useful result (as the largest entry in $h_{\rm opt}$ moved position and several other entries are not that far from the largest). 

\subsection{Optimal interpolation in multivariable Hardy and Bergman spaces where $p\in 2\BN$.}

Using ideas from \cite[Proof of Theorem 2]{Harmsen} we have the following observation. As before $\cZ^p$ is either a Hardy space $H^p$ or a Bergman space $A^p$.

\begin{proposition}\label{Bas2} Let $1<p<\infty$. 
Consider the interpolation problem in $\cZ^p$ with $f(z_i)=s_i\neq 0$, $i=1,\ldots,k$. Let the unique minimal norm interpolant be denoted by $f_{\tu{min},p}$. Associated, let $g_{\tu{min}, {\bf{s}}^{p/2}}$ denote the unique minimal interpolant in $\cZ^2$ with interpolation conditions $g(z_i)=s_i^{\frac{p}{2}}$, $i=1,\ldots, k$. If $f_{\tu{min},p}^{\frac{p}{2}}$ and $g_{\tu{min}, {\bf{s}}^{p/2}}^\frac2p 
$ are analytic, $f_{\tu{min},p}^{\frac{p}{2}}(z_i)=s_i^{\frac{p}{2}}$ and $g_{\tu{min}, {\bf{s}}^{p/2}}^\frac2p (z_i)=s_i$, $i=1,\ldots,k$, then 
$f_{\tu{min},p}=g_{\tu{min}, {\bf{s}}^{p/2}}^\frac2p$. When $p\in 2\BN$, the analyticity of $f_{\tu{min},p}^{\frac{p}{2}}$ and the interpolation conditions $f_{\tu{min},p}^{\frac{p}{2}}(z_i)=s_i^{\frac{p}{2}}$, $i=1,\ldots,k$, are automatic. 
\end{proposition}

Note that confirming that both $f_{\tu{min},p}^{\frac{p}{2}}$ and $g_{\tu{min}, {\bf{s}}^{p/2}}^\frac2p$ are indeed interpolants, is necessary as the branch of logarithm should have been chosen consistently. Indeed, it is for instance possible that the initial branch chosen to compute $s_i^{\frac{p}{2}}$, $i=1,\ldots, k$, may not work for $g_{\tu{min}, {\bf{s}}^{p/2}}^\frac2p $.

\begin{proof}
If $f_{\tu{min},p}^{\frac{p}{2}} \in \cZ^2$ then this function satisfies the interpolation conditions $g(z_i)=s_i^{\frac{p}{2}}$, $i=1,\ldots, k$. As $g_{\tu{min}, {\bf{s}}^{p/2}} \in \cZ^2$ is the minimal norm interpolant for these conditions, we obtain that 
$$ \| f_{\tu{min},p}^{\frac{p}{2}} \|_{\cZ^2} \ge \| g_{\tu{min}, {\bf{s}}^{p/2}} \|_{\cZ^2}. $$
But then
$$ \| f_{\tu{min},p} \|_{\cZ^p}^p = \| f_{\tu{min},p}^{\frac{p}{2}} \|_{\cZ^2}^2 \ge \| g_{\tu{min}, {\bf{s}}^{p/2}} \|_{\cZ^2}^2 = \| g_{\tu{min}, {\bf{s}}^{p/2}}^\frac2p \|_{\cZ^p}^p .$$
As the minimal norm interpolant is unique, we must have $f_{\tu{min},p}=g_{\tu{min}, {\bf{s}}^{p/2}}^\frac2p$.
\end{proof}

As is well-known (see, e.g., \cite[Chapter 3]{PaulsenR}), in the case of a reproducing kernel Hilbert space (RKHS) applying the first representer theorem (Theorem \ref{repr}) comes down to solving a system of linear equations. Indeed, we obtain that 
\begin{equation}\label{rkhs} x_\tu{min} = \sum_{j=1}^k c_j K(\cdot , z_j) \ \hbox{\rm where} \  (c_j)_{j=1}^k = [(K(z_j,z_i))_{i,j=1}^k]^{-1} (s_i)_{i=1}^k . \end{equation}
If we combine this with Proposition \ref{Bas2}, we are able to easily come up with a candidate for the minimal norm interpolant in $H^p$ or $A^p$ for the case when $p \in 2\BN$. Subsequently checking for analyticity (of $g_\tu{min}^{2/p}$ in Proposition \ref{Bas2}), then seals the deal. Let us illustrate this on an example. 

\begin{example}
We consider the two-variable Bergman space $A_2^4$. Let $z_1=(\frac14, \frac34)$, $z_2=(0,0)$ and $s_1=1$, $s_2=98/100$.
Proposition \ref{Bas2} tells us to consider the minimal norm interpolant $g_\tu{min}$ in $A^2_2$ with interpolation conditions $g(z_1)=1^2$ and $g(z_2)=(98/100)^2$. 
Solving for $c_1$ and $c_2$ in
$$  
\begin{bmatrix}
\frac{1}{(1-\langle z_1 , z_1 \rangle_{\rm Eucl})^3} & \frac{1}{(1-\langle z_2 , z_1 \rangle_{\rm Eucl})^3} \cr
\frac{1}{(1-\langle z_1 , z_2 \rangle_{\rm Eucl})^3} & \frac{1}{(1-\langle z_2 , z_2 \rangle_{\rm Eucl})^3}
\end{bmatrix}
\begin{bmatrix} c_1\cr c_2
\end{bmatrix} =
\begin{bmatrix}
\frac{16^3}{6^3} & 1 \cr 1 & 1
\end{bmatrix}
\begin{bmatrix}
c_1\cr c_2
\end{bmatrix} = 
\begin{bmatrix}
1\cr \frac{98^2}{100^2}
\end{bmatrix},$$
giving $c_1=2673/1212500, c_2=1161812/1212500$,
we find that
$$ g_\tu{min}(\zeta)=
\frac{1}{1212500}\left(
\frac{2673}{(1-\zeta_1/4-3\zeta_2/4)^3}+1161812\right).$$
For $\zeta=(\zeta_1,\zeta_2) \in B_2$ we have that 
$|\langle \zeta , \begin{bmatrix}
    1/4 \cr 3/4
\end{bmatrix} \rangle_{\rm Eucl}| \le \sqrt{10}/4 $ and thus 
$$|1-\zeta_1/4-3\zeta_2/4|^3 \ge |1-\frac{\sqrt{10}}{4}|^3
\approx 0.00918587. $$
Since $2673/1161812 \approx 0.0023007<0.00918587$, we have that $|g_\tu{min}(\zeta)| \ge \epsilon >0$, $\zeta \in B_2$, for some $\epsilon>0$.
Thus $g_\tu{min}^\frac12$ is a well-defined analytic function on $B_2$, and thus by Proposition \ref{Bas2} we find that the optimal interpolant is 
$f_{\tu{min},4}=g_\tu{min}^\frac12$. \hfill $\Box$
\end{example}

Finally, we observe that in the setting of Figure \ref{fig:enter-label0}, the function $f_\tu{min}^{[p]}$, $p\in [1.7,2.6]$, does not have any roots in $\BD$. Thus, by Proposition \ref{Bas2},  $f_\tu{min}^{[p]}=g_{\tu{min}, {\bf{s}}^{p/2}}^\frac2p$, explaining the continuity we see in Figure \ref{fig:enter-label0}.

\bigskip

\subsection*{Acknowledgements}
We thank Professor David Ambrose for pointing us to Henry Helson's book 'Harmonic Analysis'. 
This work is based on research supported in part by the National Research Foundation of South Africa (NRF, Grant Numbers 118513 and 127364) and the DSI-NRF Centre of Excellence in Mathematical and Statistical Sciences (CoE-MaSS). Any opinion, finding and conclusion or recommendation expressed in this material is that of the authors and the NRF and CoE-MaSS do not accept any liability in this regard. We also gratefully acknowledge funding from the National Graduate Academy for Mathematical and Statistical Sciences (NGA(MaSS)). In addition, HW is partially supported by United States National Science Foundation (NSF) grant DMS-2000037.


\begin{thebibliography}{xx}

\bibitem{Aizerman} 
M.A. Aizerman, E.M. Braverman, and L.I. Rozonoer, Theoretical foundations of the potential function method in pattern recognition learning, {\em Automat.\ Remote Control} {\bf 25} (1964), 821--837.

\bibitem{aron0} 
N. Aronszajn, La th\'eorie des noyaux reproduisants et ses applications I, {\em Proc.\ Cambridge Philos.\ Soc.} {\bf 39} (1943), 133--153.

\bibitem{aron}
N. Aronszajn, Theory of reproducing kernels, {\em Trans.\ Amer.\ Math.\ Soc.} {\bf 68} (1950), 337--404.

\bibitem{AH05}
K. Atkinson and W. Han, {\em Theoretical numerical analysis}, Texts Appl.\ Math.\ {\bf 39}, Springer, Dordrecht, 2009.


\bibitem{Axler} 
S. Axler, {\em Bergman spaces and their operators}, Pitman Res.\ Notes Math.\ Ser.\ {\bf 171}, Longman Scientific \& Technical, Harlow, 1988, pp.\ 1--50.

\bibitem{ABR} 
S. Axler, P. Bourdon, and W. Ramey, {\em Harmonic function theory}, Springer--Verlag, New York, 1992.

%\bibitem{Bach} F. R. Bach and M. I. Jordan. Kernel independent component
%analysis. Journal of Machine Learning Research, 3(1–48), 2002.

\bibitem{BK12} 
C. B{\'e}n{\'e}teau and D. Khavinson, A survey of linear extremal problems in analytic function spaces, {\em CRM Proc.\ Lecture Notes} {\bf 55},  Amer.\ Math.\ Soc., Providence, RI, 2012, pp.\ 33--46.

\bibitem{Berge} 
C. Berge, {\em Topological Spaces}, Oliver and Boyd, 1963. 

%\bibitem{Bessonov} R. V. Bessonov. Analytic approximation in  Lp  and coinvariant subspaces of the Hardy space
%J. Approx. Theory 174 (2013), 113–120.

\bibitem{Birkhoff} 
G. Birkhoff, Orthogonality in linear metric spaces, {\em Duke Math.\ J.} {\bf 1} (1935), 169--172.

\bibitem{Boser} 
B.E. Boser, I.M. Guyon, and V.N. Vapnik, A training algorithm for optimal margin classifiers, In:\ {\em Proceedings of the fifth annual workshop on Computational learning theory}, pp.\ 144--152, 1992.

%\bibitem{CKL18}
%P. Cerejeiras, U K\"{a}hler, and D. Legatiuk, Interpolation of monogenic functions by using reproducing kernel Hilbert spaces, {\em Math.\ Methods Appl.\ Sci.} {\bf 41} (2018), 8100--8114.

\bibitem{CM22}
H. Centeno and J.M. Medina, A converse sampling theorem in reproducing kernel Banach spaces, {\em Sampl.\ Theory Signal Process.\ Data Anal.} {\bf 20} (2022), Paper No.\ 8.

\bibitem{CX21}
R. Cheng and Y. Xu, Minimum norm interpolation in the  $\ell^1(\mathbb{N})$  space, {\em Anal. Appl.\ (Singap.)} {\bf 19} (2021), 21--42.

\bibitem{CMS94}
J. Cima, T. MacGregor, M. Stessin, Recapturing Functions in $H^p$ Spaces, {\em Indiana Univ.\ Math.\ J.} {\bf 43} (1994), 205--220.

%\bibitem{CS95}
%JA Cima, MI Stessin, Linearity of the nearest point cross section operators on holomorphic function spaces - Journal d'Analyse Mathematique, 1995
%
\bibitem{C90} 
J.B. Conway, {\em A Course in Functional Analysis}, Springer-Verlag, New York, 1985.

%\bibitem{Cowen} Cowen, Carl; MacCluer, Barbara, {\em Composition Operators on Spaces of Analytic Functions}, Studies in Advanced Mathematics, CRC Press, p. 27

%\bibitem{D04}
%S.S. Dragomir, {\em Semi-inner products and applications}, Nova Science Publishers, Inc., Hauppauge, NY, 2004.

\bibitem{D70}
P.L. Duren, {\em Theory of $H_p$ Spaces}, Academic Press, 1970.

\bibitem{DurenBergman} 
P.L. Duren and A. Schuster, {\em Bergman spaces}, Math.\ Surveys Monogr.\ {\bf 100},  Amer.\ Math.\ Soc., Providence, RI, 2004.

%\bibitem{DGM} Duren, Peter L.; Gallardo-Gutiérez, Eva A.; Montes-Rodríguez, Alfonso (2007-06-03), A Paley-Wiener theorem for Bergman spaces with application to invariant subspaces, vol. 39, Bulletin of the London Mathematical Society, pp. 459–466.

\bibitem{Ebenstein}
S.E. Ebenstein,  Some  $H^p$  spaces which are uncomplemented in  $L^p$, {\em Pacific J. Math.} {\bf 43} (1972), 327--339.

\bibitem{Elliott} 
S.J. Elliott and A. Wynn, Composition operators on the weighted Bergman spaces of the half plane, {\em Proc.\ Edinb.\ Math.\ Soc.\ (2)} {\bf 54} (2011), 374--379.

\bibitem{FHHMZ11}
M. Fabian, P. Habala, P. H\'{a}jek, V. Montesinos, and V. Zizler, {\em Banach space theory, The basis for linear and nonlinear analysis}, CMS Books Math./Ouvrages Math.\ SMC Springer, New York, 2011.

\bibitem{FHY15}
G. E. Fasshauer, F. J. Hickernell, and Q. Ye, Solving support vector machines in reproducing kernel Banach spaces with positive definite functions, {\em Appl.\ Comput.\ Harmon.\ Anal.} {\bf 38} (2015), 115--139.

%\bibitem{Ferguson} Timothy James Ferguson, {\em Extremal problems in Bergman spaces}, Ph. D. Thesis, University of Michigan, 2011.
 
\bibitem{Ferguson2}  
Timothy Ferguson,  Solution of extremal problems in Bergman spaces using the Bergman projection, {\em Comput.\ Methods Funct.\ Theory} {\bf 14} (2014), 35--61.

% \bibitem{Ferguson3} Ferguson, Timothy. Extremal problems in Bergman spaces and an extension of Ryabykh's Hardy space regularity theorem for  $1<p<\infty$ 
%Indiana Univ. Math. J. 66 (2017), no. 1, 259–274.

\bibitem{Fine} 
S. Fine and K. Scheinberg, Efficient SVM training using low-rank kernel representations. {\em J. Mach.\ Learn.\ Res.} {\bf 2} (2001), 243--264.

\bibitem{Fuku} 
K. Fukumizu, F.R. Bach, and M.I. Jordan, Dimensionality reduction for supervised learning with reproducing kernel Hilbert spaces. {\em J. Mach.\ Learn.\ Res.} {\bf 4} (2004), 73--99.

\bibitem{G67}
J.R. Giles, Classes of semi-inner-product spaces, {\em Trans.\ Amer.\ Math.\ Soc.} {\bf 129} (1967), 436--446.

\bibitem{Harmsen}
Bas Harmsen, {\em Interpolatieproblemen in $H_p$-ruimten}, Doctoraalscriptie, Radboud Universiteit Nijmegen, The Netherlands, 2005.

\bibitem{Gretton} 
A. Gretton, R. Herbrich, A. Smola, O. Bousquet, and B. Sch\"olkopf,
Kernel methods for measuring independence, {\em J. Mach.\ Learn.\ Res.} {\bf 6} (2005), 2075--2129.

%\bibitem{GS} R. F. Gundy and E. M. Stein, $H^p$ theory for the poly-disc
%Proc.Natl.Acad.Sci.USA Vol.76,No.3,pp.1026-1029,March1979
%Section 5  is two-variable half-spaces

\bibitem{Hedenmalm} 
H. Hedenmalm, B. Korenblum, and K. Zhu, {\em Theory of Bergman spaces}, Grad.\ Texts in Math.\ {\bf 199}, Springer--Verlag, New York, 2000.

\bibitem{Helson} 
H. Helson, {\em Harmonic Analysis}, Addison-Wesley, Reading, MA, 1983.

\bibitem{Hensgen} 
W. Hensgen, On the dual space of  $H^p(X),1<p<\infty$, {\em J. Funct.\ Anal.} {\bf 92} (1990), 348--371.

%\bibitem{HV00}
%Hollenbeck, Brian(1-MO); Verbitsky, Igor E.(1-MO)
%Best constants for the Riesz projection.(English summary)J. Funct. Anal.175(2000), no.2, 370–392.
\bibitem{Hoffman}
K. Hoffman, {\em Banach Spaces of Analytic Functions}, Prentice Hall, Englewood Cliffs, 1962.

\bibitem{istr} 
V.I. Istr\u{a}\c{t}escu, {\em Strict convexity and complex strict convexity}, Lecture Notes in Pure and Appl.\ Math. {\bf 89}, Marcel Dekker Inc., New York, 1984.

%\bibitem{JPP09} Birgit Jacob, Jonathan R. Partington and Sandra Pott,
%Tangential Interpolation in Weighted Vector-valued $H^p$ Spaces, with Applications.
%{\em Complex Analysis and Operator Theory} volume 3, Article number: 697 (2009)

\bibitem{James}
R.C. James, Orthogonality and linear functionals in normed linear spaces, {\em Trans. Amer. Math. Soc.} {\bf 61} (1947), 265--292.

%\bibitem{James1}
%R. C. James, Reflexivity and the supremum of linear functionals,{\em Ann. of Math.} 66 (1957), 159--169.
%
%\bibitem{James2}
%R. C. James, Weak compactness and reflexivity. {\em Israel J. Math.} 2, 101--119 (1964).

%\bibitem{JP} Jevtić, M., Pavlović, M. Harmonic Bergman Functions on the Unit Ball in Rn . Acta Mathematica Hungarica 85, 81–96 (1999).

%\bibitem{KMS00} Khavinson, Dmitry(1-AR); McCarthy, John E.(1-WASN); Shapiro, Harold S.(S-RIT)
%Best approximation in the mean by analytic and harmonic functions.Indiana Univ. Math. J.49(2000), no.4, 1481–1513.

\bibitem{KS97} 
D. Khavinson and M. Stessin, Certain linear extremal problems in Bergman spaces of analytic functions, {\em Indiana Univ.\ Math.\ J.} {\bf 46} (1997), 933--974.

\bibitem{K} 
S.Ya. Khavinson, Two papers on extremal problems in complex analysis, {\em Amer.\ Math.\ Soc.\ Transl.\ Ser.\ 2}, No.\ {\bf 129}, 1980. 

\bibitem{Koosis} 
P. Koosis, {\em Introduction to $H^p$ spaces}, Cambridge University Press, 1998.

\bibitem{LZZ22}
R.R. Lin, H.Z. Zhang, and J. Zhang, On reproducing kernel Banach spaces: generic definitions and unified framework of constructions, {\em Acta Math.\ Sin.\ (Engl.\ Ser.)} {\bf 38} (2022), 1459--1483.

\bibitem{LWXY23}
Q. Liu, R. Wang, Y. Xu, and M. Yan, Parameter choices for sparse regularization with the $\ell^1$ norm, {\em Inverse Problems} {\bf 39} (2023), Paper No.\ 025004.

\bibitem{L61}
G. Lumer, Semi-inner-product spaces, {\em Trans.\ Amer.\ Math.\ Soc.} {\bf 100} (1961), 29--43.

\bibitem{MN} X. Ma, C.L. Nikias, Joint estimation of time delay and frequency delay in impulsive noise using fractional lower order statistics, {\em IEEE Trans.\ Signal Process.} {\bf 44} (1996), 2669--2687.

\bibitem{MR50} 
A.J. Macintyre and W.W. Rogosinski, Extremum problems in the theory of analytic functions, {\em Acta Math.} {\bf 82} (1950), 275--325.

\bibitem{M98}
R.E. Megginson, {\em An introduction to Banach space theory}, Grad.\ Texts in Math.\ {\bf 183}, Springer-Verlag, New York, 1998.

%\bibitem{MS08}
%P.S. Muhly and B. Solel, Schur class operator functions and automorphisms of Hardy algebras, {\em Doc.\ Math.} {\bf 13} (2008), 365--411.


%\bibitem{PN21}
%R. Parhi, R.D. Nowak, Banach space representer theorems for neural networks and ridge splines, {\em J. Mach.\ Learn.\ Res.} {\bf 22} (2021), Paper No.\ 43.

\bibitem{PaulsenR} 
V.I. Paulsen and M. Raghupathi, {\em An introduction to the theory of reproducing kernel Hilbert spaces}, Cambridge Stud.\ Adv.\ Math.\ {\bf 152}, Cambridge University Press, Cambridge, 2016.

\bibitem{Ramey} 
W.C. Ramey and H. Yi,  Harmonic Bergman functions on half--spaces, {\em Trans. Amer. Math. Soc.} {\bf 348} (1996), 633--660.

%\bibitem{Riesz} M. Riesz, Sur les fonctions conjug{\'e}es, {\em Math. Z.} 27 (1927), 218-244.

\bibitem{RS} 
W.W. Rogosinski and H.S. Shapiro, On certain extremum problems for analytic functions, {\em Acta Math.} {\bf 90} (1953), 287--318.

\bibitem{Rudin2} 
W. Rudin, {\em Function theory in Polydisks}, W.A. Benjamin Inc., New York, NY, 1969.

\bibitem{Rudin} 
W. Rudin, {\em Function theory in the unit ball of  $\BC^n$},  Classics Math.\ Springer-Verlag, Berlin, 2008.

% \bibitem{S21}
% K. Schlegel, When is there a representer theorem? Reflexive Banach spaces, {\em Adv.\ Comput.\ Math.} {\bf 47} (2021), no.\ 4,  Paper No.\ 54.

%\bibitem{S19}
%K. Schlegel,  When is there a representer theorem? Nondifferentiable regularisers and Banach spaces, {\em J. Global Optim.} {\bf 74} (2019), 401--415.

\bibitem{SHS01}
B. Sch\"{o}lkopf, R. Herbrich, and A.J. Smola, A Generalized Representer Theorem, In:\ {\em Computational Learning Theory}, Lecture Notes in Computer Science. Vol. 2111. Berlin, Heidelberg:\ Springer, 2001, pp.\ 416--426.

\bibitem{Slavakis} 
K. Slavakis, P. Bouboulis, and S. Theodoridis, Chapter 17--Online learning in reproducing kernel Hilbert spaces, in:\  {\em Academic Press Library in Signal Processing: Volume 1, Signal Processing Theory and Machine Learning}, p.\ 883--987. Elsevier, 2014.

%\bibitem{SZH}
%G. Song, H. Zhang, and F.J. Hickernell,  Reproducing kernel Banach spaces with the $\ell^1$ norm, {\em Appl.\ Comput.\ Harmon.\ Anal.} {\bf  34} (2013),  96--116.

\bibitem{SFL11} 
B.K. Sriperumbudur, K. Fukumizu, and G.R.G. Lanckriet, Universality, characteristic kernels and RKHS embedding of measures, {\em J. Mach. Learn. Res.} {\bf 12} (2011), 2389--2410.

%\bibitem{SFL24} Bharath K.
%Sriperumbudur, Kenji Fukumizu, and Gert R. G. Lanckriet. 
%Learning in Hilbert vs. Banach Spaces: A Measure Embedding Viewpoint. Preprint.

\bibitem{SW} 
E.M. Stein and G. Weiss, {\em Introduction to Fourier Analysis on Euclidean Spaces}, Princeton Math.\ Ser.\ {\bf 32}, Princeton University Press, Princeton, NJ, 2016.

%\bibitem{Stein} I. Steinwart. On the influence of the kernel on the consistency of
%support vector machines. {\em Journal of Machine Learning Research}
%2 (2001), 67–93.

%\bibitem{S99} Michael I. Stessin (1999) Minimal interpolation in spaces of analytic functions, Complex Variables, Theory and Application: An International Journal, 38:1, 47-68

\bibitem{S-NK} 
B\'ela Sz.-Nagy and Adam Kor\'anyi, Operatortheoretische Behandlung und Verallgemeinerung eines Problemkreises in der komplexen Funktionentheorie, {\em Acta Math.} {\bf 100} (1958), 171–202.

\bibitem{U21}
M. Unser,  A unifying representer theorem for inverse problems and machine learning, {\em Found.\ Comput.\ Math.} {\bf 21} (2021), 941--960.

% \bibitem{VW} 
% S. Voronin and Hugo J. Woerdeman, A new iterative firm-thresholding algorithm for inverse problems with sparsity constraints, {\em Appl.\ Comput.\ Harmon.\ Anal.}   {\bf 35} (2013), 151–164.

\bibitem{X} 
Y. Xu, Sparse machine learning in Banach spaces, {\em Appl.\ Numer.\ Math.} {\bf 187} (2023), 138--157.

\bibitem{XY19}
Y. Xu and Q. Ye,  Generalized Mercer kernels and reproducing kernel Banach spaces, {\em Mem.\ Amer.\ Math.\ Soc.} {\bf 258} (2019), no.\ 1243.

\bibitem{W90}
G. Wahba, {\em Spline Models for Observational Data}, CBMS-NSF Regional Conference Series in Applied Mathematics vol. 59, SIAM, Philadelphia, 1990.

\bibitem{WX} 
R. Wang and Y. Xu, Representer theorems in Banach spaces:\ minimum norm interpolation, regularized learning and semi-discrete inverse problems, {\em  J. Mach.\ Learn.\ Res.} {\bf 22} (2021), Paper no.\ 225.

\bibitem{WXY} 
R. Wang, Y. Xu, and M. Yan, Sparse representer theorems for learning in reproducing kernel Banach spaces, {\em J. Mach.\ Learn.\ Res.} {\bf 25} (2024), Paper no.\ 93.

\bibitem{ZSZ} 
W.-J. Zeng, H.C. So, and A.M. Zoubir, An $\ell_p$-norm minimization approach to time delay estimation in impulsive noise, {\em Digit.\ Signal Process.} {\bf 23} (2013), 1247--1254.

\bibitem{ZXZ09}
H. Zhang, Y Xu, and J. Zhang, Reproducing kernel Banach spaces for machine learning, {\em J. Mach.\ Learn.\ Res.} {\bf 10} (2009), 2741--2775.


\bibitem{ZZ13}
H. Zhang and J. Zhang, Vector-valued reproducing kernel Banach spaces with applications to multi-task learning, {\em J. Complexity} {\bf 29} (2013), 195--215.

\bibitem{ZZ11}
H. Zhang and J. Zhang, Frames, Riesz bases, and sampling expansions in Banach spaces via semi-inner products, {\em Appl.\ Comput.\ Harmon.\ Anal.} {\bf 31} (2011), 1–-25.

%\bibitem{Zhu} Kehe Zhu, {\em Operator theory in function spaces}, Mathematical surveys and monographs,  v. 138, American Mathematical Society, 2007.

\end{thebibliography}
\end{document}